\documentclass[11pt]{article}

\usepackage{amsthm,amsfonts,amssymb,amsmath,color, mathtools}
\usepackage{indentfirst}
\usepackage{enumerate}
\usepackage{cleveref}
\allowdisplaybreaks[4]

\usepackage{tikz}
\usepackage[title]{appendix}
\numberwithin{equation}{section}

\renewcommand\a{\alpha}

\newcommand\R{\mathbb R}\newcommand\N{\mathbb N}

\def\de{\delta}

\def\O{\Omega}

\def\epsilon{\varepsilon}
\def\e{\varepsilon}

\newcommand\br{\begin{rem}}
\newcommand\er{\end{rem}}
\newcommand\bp{\begin{pmatrix}}
\newcommand\ep{\end{pmatrix}}
\newcommand\be{\begin{equation}}
\newcommand\ee{\end{equation}}
\newcommand\ba{\begin{equation}\begin{aligned}}
\newcommand\ea{\end{aligned}\end{equation}}

\newcommand\nn{\nonumber}

\newcommand{\bfvarphi}{\boldsymbol{\varphi}}

\usepackage{layout}

\usepackage{amssymb}

\newcommand{\uu}{{\mathbf u}}
\newcommand{\vv}{{\mathbf v}}
\newcommand{\ff}{{\mathbf f}}

\newcommand{\vu}{\vc{u}}

\newcommand{\vc}[1]{{\bf #1}}

\newcommand{\dx}{{\rm d} {x}}
\newcommand{\dt}{{\rm d} t }

\newcommand{\dive}{{\rm div\,}}

\newtheorem{defi}{Definition}[section]
\newtheorem{theorem}[defi]{Theorem}
\newtheorem{proposition}[defi]{Proposition}
\newtheorem{lemma}[defi]{Lemma}

\newtheorem{remark}[defi]{Remark}

\def\thesection{\arabic{section}}

\numberwithin{equation}{section}

\begin{document}

\title{Qualitative and quantitative homogenization of some non-Newtonian flows in perforated domains: case of `small holes'}

\author{Yong Lu \footnote{School of Mathematics, Nanjing University, Nanjing 210093, China, luyong@nju.edu.cn}\and Zhengmao Qian \footnote{School of Big Data and Statistics, Tongling University, Tongling 244000, China, zmqian0302@163.com}\and Chenchen Zhang\footnote{School of Mathematics, Nanjing University, Nanjing 210093, China, chenchenzhang@smail.nju.edu.cn}}
\date{}

\maketitle

\begin{abstract}
We consider the homogenization of three dimensional viscous incompressible non-Newtonian flows satisfying certain general $r$-structure in perforated domains. We focus on the case of `small holes' by assuming the holes under consideration are of size $\e^{\a}$ with $\a >3$, where $\e$ is the perforation parameter used to measure the mutual distance between the  holes. We show the limit equations remain unchanged in the homogenization limit under the constraint $ \frac{6(\alpha - 1)}{4\alpha-5}< r<3-\frac{3}{\alpha}$, which seems optimal in the sense of Sobolev capacity of holes as explained in Remark \ref{rem-thm1}.  Quantitative convergence rates are further derived for both the velocity field and the pressure. To the best of our knowledge, both the qualitative and quantitative homogenization results are firstly given for non-Newtonian flows in the `small holes' case.  
\end{abstract}

\renewcommand{\refname}{References}


\section{Introduction}
Homogenization problems in fluid mechanics represent the study of fluid flows in domains perforated with a large number of small holes. The goal is to describe the asymptotic behavior of fluid flows as the number of holes goes to infinity and the size of holes goes to zero simultaneously. The target equations that describe the limit behavior of fluid flows are called homogenized equations, which are defined in homogeneous domains without holes.

In this paper, we consider the homogenization of incompressible viscous non-Newtonian flows in perforated domains. Non-Newtonian fluids are extensively used in a number of applied problems involving the production of oil and gas from underground reservoirs. We introduce the following equations describing steady non-Newtonian fluids in three dimensions:

\begin{equation}\label{1.1}
\begin{cases}
-{\rm div}\, \mathbb{S}(D\mathbf{u}_\varepsilon) +  (\mathbf{u}_\varepsilon\cdot \nabla) \mathbf{u}_\varepsilon +\nabla p_\varepsilon=\mathbf{f},& {\rm in}\;\Omega_\varepsilon, \\
{\rm div} \,\mathbf{u}_\varepsilon=0,& {\rm in}\;\Omega_\varepsilon,\\
\mathbf{u}_\varepsilon=0,& {\rm on}\;\partial\Omega_\varepsilon.
\end{cases}
\end{equation}
Here $\mathbf{u}_\varepsilon$ is the velocity, $\nabla \mathbf{u}_\varepsilon$ is the velocity gradient tensor, $D\mathbf{u}_\varepsilon=\frac{1}{2}(\nabla \mathbf{u}_\varepsilon+\nabla^T\mathbf{u}_\varepsilon)$ denotes the rate-of-strain tensor, $p_\varepsilon$ denotes the pressure, $\mathbf{f}$ is the density of the external force and is assumed to be independent of $\e$ and belongs to $L^{\infty}(\Omega;\R^{3})$. By rather straightforward modifications, our argument still works and our main results still hold for $\mathbf{f}$ depending on $\e$ and converging strongly in some $L^{q}(\Omega;\R^{3})$ space. 

Inspired by the previous studies (see \cite{DRW10} or \cite{BGMS12}) concerning the existence theory, we impose the following general assumptions on the stress tensor. We say that the stress tensor $\mathbb{S}$ possesses a {\em $r$-structure} if there exists $r\in (1,\infty)$ and $\delta\geq 0$, such that

\begin{itemize}
\item  

$\displaystyle \mathbb{S}: \mathbb{M}^3_{sym}\to \mathbb{M}^3_{sym}$   is  continuous, where  $\mathbb{M}^3_{sym}$ denotes the collection of $3\times 3$ symmetric matrices. 

\item 
Growth condition: there exists $c_{0}>0$ such that 
\ba\label{growth}
|\mathbb{S}(\boldsymbol{\xi})|\leq c_0(\delta+|\boldsymbol{\xi}|)^{r-2}|\boldsymbol{\xi}|
\ea
\ {\rm for  all}  $\boldsymbol{\xi}\in \mathbb{M}^3_{sym}$;

\item 
Coercivity: there exists $c_{1}>0$ such that 
\ba\label{coer}
\mathbb{S}(\boldsymbol{\xi}):\boldsymbol{\xi}\geq c_1(\delta+|\boldsymbol{\xi}|)^{r-2}|\boldsymbol{\xi}|^2
\ea
\  {\rm for all} $\boldsymbol{\xi}\in \mathbb{M}^3_{sym}$;

\item 
Strict monotonicity:
\ba\label{monotonicity}
(\mathbb{S}(\boldsymbol{\xi})-\mathbb{S}(\boldsymbol{\zeta})):(\boldsymbol{\xi}-\boldsymbol{\zeta})>0
\ea
\  {\rm for  all}  $\boldsymbol{\xi},  \boldsymbol{\zeta}\in \mathbb{M}^3_{sym}$ \ {\rm such  that}  $\boldsymbol{\xi}\neq \boldsymbol{\zeta}$.
\end{itemize}

    Typical examples satisfying \eqref{growth}--\eqref{monotonicity} include the power-law and the Carreau-Yasuda law:
    \ba\label{power-CY}
        \mathbb{S}(\boldsymbol{\xi}) & =\mu_0(\delta+|\boldsymbol{\xi}|)^{r-2}\boldsymbol{\xi}+\mu_1\boldsymbol{\xi},
        \\
        \mathbb{S}(\boldsymbol{\xi}) & =\mu_0(\delta^2+|\boldsymbol{\xi}|^{2})^{\frac{r-2}{2}}\boldsymbol{\xi}+\mu_1\boldsymbol{\xi},
    \ea
    with $1<r<\infty, \ \delta\geq 0$, $\mu_0>0,\ \mu_1\geq 0$.  The fluid is said to be shear thinning if $1<r<2$ and to be shear thickening if $r>2$.
    
    \medskip

    The perforated domain $\Omega_{\e}$ under consideration is described as follows. Let $\Omega\subset \mathbb{R}^3$ be a bounded domain of class $C^{2,1}$ and $\{T_{\varepsilon,k}\}_{k\in K_\e}\subset\Omega$ be a family of closed sets called holes satisfying
    \begin{equation}\label{hole description}
    T_{\varepsilon,k} = x_{\e, k}+a_\e T\subset B( x_{\e, k},\delta_1a_\e)\subset B( x_{\e,k},\delta_2 \e) \subset \varepsilon Q_k,
    \end{equation}
    where $Q_k=(-\frac{1}{2},\frac{1}{2})^3+k,$ $x_{\e,k}\in \e Q_k,\; k\in {\mathbb{Z}}^3$ and $a_\e=\e^\alpha$ for some $\alpha\geq1$ independent of $\e$; $\delta_i,$ $i=1,2$ are fixed positive numbers. $T\subset\mathbb{R}^3$ is a model hole which is assumed to be a closed simply connected domain contained in $Q_0$ with $C^{2,1}$ boundary and contains the origin $0$ as an interior point. The parameters $\e$ and $a_\e$ are used to measure the mutual distance and the size of the holes, respectively.  Throughout the paper, the perforated domain $\Omega_\varepsilon$ is then defined as:
    \begin{equation}\label{def-O}
        \Omega_\varepsilon=\Omega\backslash\bigcup_{k\in K_\varepsilon}T_{\varepsilon,k},\quad {\rm where} \ K_\varepsilon=\{k\in {\mathbb{Z}}^3:\varepsilon\overline{Q}_{k}\subset\Omega\}.
    \end{equation}
    Note that the periodicity of the distribution of holes is not assumed in this paper. By the distribution of holes assumed above, the number of holes in $\O_\e$ is of order $\e^{-3}$. 
    
    \medskip

    We review some known results in the mathematical study of homogenization problems in fluid mechanics. The homogenization of the Stokes system in perforated domains has been systematically studied. In \cite{Tartar80}, Tartar considered the case where the size of the holes is proportional to the mutual distance of the holes and Darcy's law was derived.
    Then Allaire \cite{Allaire90.1, Allaire90.2} considered general cases and showed that the homogenized equations are determined by the ratio $\sigma_\varepsilon$ between the size and the mutual distance of the holes:
    \ba\nonumber
        \sigma_\varepsilon=\left(\frac{\varepsilon^d}{a_\varepsilon^{d-2}}\right)^\frac{1}{2},\quad d\geq3;
        \qquad
        \sigma_\varepsilon=\varepsilon\big | {\rm log}\frac{a_\varepsilon}{\varepsilon}\big | ^\frac{1}{2},\quad d=2,
    \ea
    where $\varepsilon$ and $a_\varepsilon$ are used to measure the mutual distance of the holes and the size of the holes. In particular, if $\lim_{\varepsilon\rightarrow 0}\sigma_\varepsilon=0$ corresponding to the case of {\em large holes}, the homogenized system is Darcy's law; if $\lim_{\varepsilon\rightarrow 0}\sigma_\varepsilon=\infty$ corresponding to the case of {\em small holes}, the homogenized equations remain the same as the original ones; if $\lim_{\varepsilon\rightarrow 0}\sigma_\varepsilon=\sigma_\ast\in (0,+\infty)$ corresponding to the case of {\em critical size of holes}, the homogenized equations satisfy Brinkman's law --- a combination of Darcy's law and the original Stokes equations.

    Later, the homogenization study was extended to more complicated models describing fluid flows: Mikeli{\'c} \cite{Mikelic91} considered  the non-stationary incompressible Navier-Stokes equations; Masmoudi \cite{Masmoudi02} considered  the compressible Navier-Stokes equations; Feireisl, Novotn{\'y} and Takahashi \cite{FNT10} considered  the Navier-Stokes-Fourier equations. In these studies, only the case where the size of the holes is proportional to the mutual distance of the holes is considered and Darcy's law is recovered in the limit.

    Recently, cases with different sizes of holes are studied. In \cite{DFL17,FL15,LS18}, with collaborators the first author considered the case of small holes for the compressible Navier-Stokes equations and it is shown that the homogenized equations remain the same as the original ones.  The homogenization of compressible Navier-Stokes equations was further considered: Oschmann and Pokorn{\' y} \cite{OP23}  considered the case of small holes for the unsteady compressible Navier-Stokes equations for adiabatic exponent $\gamma>3$ which improved the condition $\gamma>6$ in \cite{LS18}; Ne{\v c}asov{\'a} and Oschmann \cite{NO23} studied the homogenization of two-dimensional evolutionary compressible Navier-Stokes equations with very small holes; H{\"o}fer, Kowalczyk and Schwarzacher \cite{HKS21} studied the case of large holes for the compressible Navier-Stokes equations at low Mach number and derived the Darcy's law; Bella and Oschmann \cite{BO22} studied the case with critical size of holes for the compressible Navier-Stokes equations at low Mach number and they derived the incompressible Navier-Stokes equations with Brinkman term; Bella, Feireisl and Oschmann \cite{BFO23} considered the case of unsteady compressible Navier-Stokes equations at low Mach number under the assumption $\O_{\e}\to\O$ in the sense of Mosco's convergence and they derived the incompressible Navier-Stokes equations. 
    
    The majority of studies assumes the small holes are periodically distributed, while more general distribution of holes was also studied: Feireisl, Namlyeyeva and Ne{\v c}asov{\'a} \cite{FNN16} studied the incompressible Navier-Stokes equations for the case with critical size of holes which impose certain mild hypotheses concerning the shape of the holes and their mutual distance, and they derived the Brinkman's law; Bella and Oschmann \cite{BO23} considered the homogenization of compressible Navier-Stokes equations for the case of randomly perforated domains with small size of holes and they derived the same limiting equations. In \cite{hofer-necasova-oschmann-2024}, H{\"o}fer, Ne{\v c}asov{\'a} and Oschmann prove quantitative homogenization toward Darcy’s law via a relative energy method for the compressible Navier–Stokes equations in periodically perforated domains with $a_{\e}\sim \e$.

    However, there are not many mathematical studies concerning the homogenization of non-Newtonian flows. Bourgeat and Mikeli\'{c} \cite{BM96} considered the stationary case of Carreau-Yasuda type flows under the assumption $a_{\e}\sim \e$ and derived Darcy's law. Mikeli\'{c} summarized some theory of stationary non-Newtonian flows in Chapter 4 of \cite{Hornung97}. In \cite{LQ24}, the first two authors considered the evolutionary case of Carreau-Yasuda law flows under the assumption $a_{\e}\sim \e$ and derived Darcy's law. H\"ofer, Lu and Oschmann in \cite{lu-oschmann-2024} considered the three dimensional non-Newtonian flows of Carreau-Yasuda type under the assumption $a_{\e}\sim \e^{\alpha}$ with $1<\alpha<\frac{3}{2}$ and derived Darcy's law for both the steady and the evolutionary cases.
    
    In this paper, we consider general steady non-Newtonian fluids for the case of `small holes', meaning that $a_{\e} = \e^{\alpha}$ with $\a>3$, and show that the limiting equations remain unchanged. To our knowledge, this is the first result concerning the homogenization of non-Newtonian flows for the case of `small holes'. In particular, quantitative convergence rates are given for the first time. 

    \subsection{Notations}
    We recall some notations. Let $1\leq r\leq \infty$ and $\Omega$ be a bounded domain. We use the notation $L_0^r(\Omega)$ to denote the space of $L^r(\Omega)$ functions with zero mean value:
    $$L_0^r(\Omega)=\Big\{f\in L^r(\Omega) \ : \ \int_\Omega f\, {\rm d}x=0\Big\}.$$
    We use $W^{1,r}(\Omega)$ to denote the classical Sobolev space, and $W^{1,r}_{0}(\Omega)$ denotes the completion of $C_c^\infty(\Omega)$ in $W^{1,r}(\Omega)$. Here $C_c^\infty(\Omega)$ is the space of smooth functions compactly supported in $\Omega$. We use $W^{-1,r'}(\Omega)$ to denote the dual space of $W^{1,r}_{0}(\Omega)$. For any $f\in L^{r}(\Omega_{\e})$, we use $\tilde f$ to denote its zero extensio:
    $$
    \tilde f  = f \ \mbox{in} \ \Omega_{\e} , \quad     \tilde f  = 0 \ \mbox{in} \  \R^{3} \setminus \Omega_{\e} 
    $$

    Now we introduce the definition of finite energy weak solutions to \eqref{1.1}:
    \begin{defi}\label{def-weak}
        Let $1<r<\infty$. We say that $(\mathbf{u}_{\varepsilon},p_{\e})$ is a finite energy weak solution of \eqref{1.1} in $ \Omega_{\e}$ provided
    \begin{itemize}
    \item $\uu_{\e} \in W^{1,r}_{0}(\Omega_{\e};\R^{3}), \ \dive \uu_{\e} = 0;$  $p_{\e} \in L^{q}(\Omega_{\e})$ for some $q>1$ depending on $r$;

    \item there holds the integral equality
    \ba\label{weak-eq-e}
        \int_{\Omega_{\varepsilon}} \big( -\mathbf{u}_{\e}\otimes \mathbf{u}_{\e}:\nabla\bfvarphi+\mathbb{S}(D\mathbf{u}_\e):D\bfvarphi   - p_{\e} \dive \bfvarphi\big)\, \dx =\int_{\Omega_\varepsilon}\mathbf{f}\cdot\bfvarphi \, \dx
    \ea
    for any test function\;$\bfvarphi\in C_c^\infty(\Omega_\varepsilon;\mathbb{R}^3);$

    \item there holds the energy inequality
    \ba\label{energy-eq}
        \int_{\Omega_\varepsilon}\mathbb{S}(D\mathbf{u}_\e):D\mathbf{u}_\e\,\dx\leq \int_{\Omega_\varepsilon}\mathbf{f}\cdot \mathbf{u}_{\varepsilon}\,\dx.
    \ea
\end{itemize}
\end{defi}

If the stress tensor satisfies the {\em $r$-structure} \eqref{growth}$-$\eqref{monotonicity} with $r>\frac{6}{5}$, the existence of weak solutions is known, see for example \cite{DRW10} or \cite{BGMS12}.  The goal of this paper is to understand the behavior of such weak solutions as the perforation parameter $\e \to 0.$

\medskip

For brevity we use $C$ to denote a positive constant independent of $\varepsilon$ throughout the paper, while the value of $C$ may differ from line to line.

\subsection{Main results}
We now state our main results, where the limits are taken up to possible extractions of subsequences.

\begin{theorem}[qualitative homogenization]\label{thm-1}
Let $\alpha>3$ and $ \frac{6(\a - 1)}{4\alpha-5}< r<3-\frac{3}{\alpha}$. Assume the stress tensor $\mathbb{S}$ satisfies the {\em $r$-structure} \eqref{growth}--\eqref{monotonicity}. Let $(\mathbf{u}_\varepsilon,p_\varepsilon)$ be a finite energy weak solution of \eqref{1.1} in the sense of Definition \ref{def-weak} and $(\tilde \uu_\e,\tilde{p}_\varepsilon)$ be their zero extensions in $\Omega$. Then there holds
    \ba\nonumber
        \tilde \uu_\e \rightarrow \mathbf{u} \ weakly \ in \ W^{1,r}_0(\Omega ;\mathbb{R}^3),\quad 
        \tilde p_\e \rightarrow p \ weakly \ in \ 
        L^{\min\{r',\frac{r^*}{2}\}}(\Omega),
    \ea
    where $r' = \frac{r}{r-1}$ and $r^{*} =  \frac{3r}{3-r}$ are the Lebesgue  conjugate exponent and Sobolev embedding exponent of $r$, respectively. 
    Moreover, the limit $(\mathbf{u},p)$ satisfies the same fluid equations in the homogeneous domain $\Omega$:
    \begin{equation}\label{limit fluid equation}
        -{\rm div}\, \mathbb{S}(D\mathbf{u}) +  (\mathbf{u}\cdot \nabla) \mathbf{u}
        +\nabla p=\mathbf{f}  , \quad \dive \vu = 0 \quad  in\ \mathcal{D'} (\Omega).
    \end{equation}
\end{theorem}

\medskip

\begin{remark}\label{rem-thm1}
We give some remarks concerning our main results and the main ideas of the proofs:

\begin{itemize}
    
   \item The constraint $ \frac{6(\a - 1)}{4\alpha-5}< r<3-\frac{3}{\alpha}$ in Theorem \ref{thm-1} implies  $r >\frac{3}{2}$. The results in \cite{DRW10} and \cite{BGMS12} ensure the existence of weak solutions in the sense of Definition \ref{def-weak}.
  
    \item The range $ \frac{6(\a - 1)}{4\alpha-5}< r<3-\frac{3}{\alpha}$ appears to be optimal in order to show the limit equations remain unchanged. Owing to the singular nature of the perforated domain $\Omega_{\e}$, the only uniform bounds available are those derived from the energy inequality, which ensures that $\uu_{\e}$ is uniformly bounded in $W^{1, r}(\Omega_{\e})$. Consequently, we obtain only the uniform boundedness of $ \mathbb{S}(D\tilde{\mathbf{u}}_\e)  $ in $L^{r'}(\Omega)$ and of $\tilde{\mathbf{u}}_\e\otimes \tilde \vu_{\e} $ in $L^{\frac{r^*}{2}}$.

    Showing the limit equations in $\Omega$ requires the zero extension $ \tilde{\bf u}_\e $ to satisfy the original equations in $\Omega$ up to a small error in the distributional sense. Since $\uu_{\e}$ only satisfies the equations in $\Omega_{\e}$,  it requires that the corresponding Sobolev capacities of $\Omega \setminus \Omega_{\e}$ to vanish (see \eqref{cap-holes}):
    \begin{itemize}
        \item for the term ${\rm div}\,\big(\mathbb{S}(D\tilde{\mathbf{u}}_\e) \big)$, given the uniform $L^{r'}(\Omega)$ estimate for $\mathbb{S}(D\tilde{\mathbf{u}}_\e)$, it is necessary to have the $r$-capacity of $\Omega \setminus \Omega_\varepsilon$ converge to zero. This yields the constraint $r<3-\frac{3}{\a}$;
        \item for the term ${\rm div}(\tilde{\mathbf{u}}_\e\otimes \tilde \vu_{\e})$, using the $L^{\frac{r^*}{2}}$ bound for $\tilde{\mathbf{u}}_\e$, it is necessary to have the $(\frac{r^*}{2})'$-capacity of $\Omega \setminus \Omega_\varepsilon$ converge to zero. This yields the constrain $r>\frac{6(\alpha-1)}{4\alpha-5}$. 
    \end{itemize}
    Consequently, the restriction $\frac{6(\alpha-1)}{4\alpha-5} < r < 3 - \frac{3}{\alpha}$ appears to be optimal. In particular, as $\a\to \infty$,  meaning that the holes become extremely small, this restriction can be relaxed to any $\frac 32 < r < 3$.
    
\item  We employ  a Bogovskii type operator in perforated domains introduced in \cite{DFL17}  to deduce the uniform estimate in $L^{\min\{r',\frac{r^*}{2}\}}$ for the pressure term $p_{\e}$, for which  $\nabla p_\varepsilon$ is uniformly bounded in $W^{-1, \min\{r',\frac{r^*}{2}\}}$ due to the fluid equations.
 
    \item Establishing the strong convergence of the nonlinear stress tensor is crucial. We first show that the momentum equations are satisfied in the homogeneous domain $\Omega$ up to a small error. We then employ the {\em solenoidal Lipschitz truncation} of Sobolev functions to enhance the regularity of the test functions and apply Minty's trick to derive the strong convergence of the velocity gradient. 

\end{itemize}
\end{remark}

\medskip

    Under additional smallness assumption on $\mathbf{f}$, the following quantitative homogenization result can be established. We focus on stress tensors taking the classical forms in \eqref{power-CY}; this could be generalized by imposing certain general assumptions. We shall focus on the non-Newtonian effect in the homogenization process and omit the Newtonian part by taking $\mu_{1} = 0$ in \eqref{power-CY}. 
    \begin{theorem}[quantitative homogenization]\label{convergence thm}
   Let $\mathbb{S}(\boldsymbol{\xi})$ be one of the forms in \eqref{power-CY} with $\delta>0,  \  \mu_0>0, \  \mu_{1} = 0$. Let $\alpha>3$ and $\frac{6(\a-1)}{4\a-5} < r<3-\frac{3}{\alpha}$.  Let $(\mathbf{u}_\varepsilon,p_\varepsilon)$ be a finite energy weak solution of \eqref{1.1} in the sense of Definition \ref{def-weak} and $(\tilde \uu_\e,\tilde{p}_\varepsilon)$ be its zero extension in $\Omega$.  Then there exists a sufficiently small constant $\kappa = \kappa(\mu_0, r,\delta, \Omega)$ such that if the external force $\|\mathbf{f}\|_{L^\infty(\O)} \le \kappa$, there holds the following improved regularity and convergence rates:
        \begin{itemize}     
        \item       if $2\leq r<3-\frac{3}{\a}$,   then $\mathbf{u}\in W^{1,2r+2}(\O)$ and
        \ba\label{rate-0r>2}
           \|\tilde{\bf u}_{\e}-{\bf u}\|_{W^{1,r}({\O})} +  \|\tilde{p}_\e-p\|_{L^{r'} (\Omega)}  +  \|\tilde{\bf u}_{\e}-{\bf u}\|^{\frac 2 r}_{W^{1,2}({\O})}  \le  C  \e^{\frac{(3-r)\alpha-3}{r^{2}}};
        \ea 

       \item  if $\frac{9}{5}\leq r<2$, then $\mathbf{u}\in W^{1,\infty}(\O)$ and
        \ba\label{rate-095<r<2}
            \|\tilde{\bf u}_{\e}-{\bf u}\|_{W^{1,r}({\O})} +   \|\tilde{p}_\e-p\|_{L^{r'}(\O)}^{\frac{1}{r-1}}\le C  \e^{\frac{(3-r)\alpha-3}{2r}};
        \ea
        \item  if $\frac{6(\a-1)}{4\a-5}<r<\frac{9}{5}$, then $\mathbf{u}\in W^{1,\infty}(\O)$ and
        \ba\label{rate-0r<95}
            \|\tilde{\bf u}_{\e}-{\bf u}\|_{W^{1,r}({\O})} +    \|\tilde{p}_\e-p\|_{L^{\frac{r^*}{2}}(\O)} ^{\frac{1}{r-1}} \le C  \e^{\frac{(4\a-5)r-6(\a-1)}{2r}}.
        \ea
        \end{itemize}
%
    \end{theorem}

    \begin{remark}  \label{rem-quant}     Regarding the quantitative homogenization result, we give some remarks:
    \begin{itemize}
    
    \item The convergence rates are crucially related on the Sobolev capacity of the holes. Indeed,  there holds the estimate:
    \ba\label{cap-holes}
    {\rm Cap}_{r} (\Omega\setminus \Omega_{\e})^{\frac{1}{r}} \leq C \e^{\frac{(3-r)\a - 3}{r}}, \quad  {\rm Cap}_{(\frac{r^*}{2})'}(\Omega\setminus \Omega_{\e}) ^{\frac{1}{(\frac{r^*}{2})'}} \leq C \e^{\frac{(4\a-5)r-6(\a-1)}{r}}.
    \ea
    
    \item  Without the smallness assumption $\|\mathbf{f}\|_{L^\infty(\O)} \le \kappa$,  the improved regularity  $\mathbf{u}\in W^{1,2r+2}(\O)$ still holds for the case $2\leq r<3-\frac{3}{\a}$ (see \cite{BKR11});  while for the case $\frac{6(\a-1)}{4\a-5}<r<2$, such smallness assumption is needed to show $\mathbf{u}\in W^{1,\infty}(\O)$ (see \cite{G12}). However, to obtain the convergence rates in \eqref{rate-0r>2}, \eqref{rate-095<r<2} or \eqref{rate-0r<95}, such smallness assumption is always needed in our analysis.
        
        \end{itemize}    
          \end{remark}

The rest of the paper is organized as follows. In Section \ref{Preliminaries}, we present several known propositions that will be used in our analysis.  In Section \ref{Uniform estimates}, we derive uniform estimates for the velocity and the pressure. In Section \ref{Momentum equations}, we construct appropriate test functions to derive the homogenized equations in the homogeneous domain without holes. In Section \ref{Strong convergence of the velocity gradient}, we employ the method of solenoidal Lipschitz truncations of Sobolev functions to obtain the strong convergence of the stress tensor and prove Theorem \ref{thm-1}. Finally, in Section \ref{Convergence Rate}, we derive the convergence rates and complete the proof of Theorem \ref{convergence thm}.
 
\section{Preliminaries}\label{Preliminaries}
\subsection{Inverse of divergence}
    To obtain estimates for $p_\e$, we use the equations to deduce the corresponding bounds for $\nabla p_\e$ and then employ the Bogovskii operator to derive uniform estimates for $p_\e$. To this end, we recall the classical Bogovskii operator, which can be found in \cite{FN09}, as well as the Bogovskii-type operator in perforated domains constructed in \cite{DFL17}:
    \begin{proposition}\label{bogov}
        Let $\O$ be a bounded Lipschitz domain and $q\in (1, \infty)$. Then there exists a linear operator
        \ba\nonumber
            \mathcal{B}: L_0^q(\O) \to W_0^{1,q}(\O;\mathbb{R}^3)
        \ea
        such that for arbitrary $f\in L_0^q(\O)$ there holds  
        \ba\nonumber
        \operatorname{div}\mathcal{B}(f)=f \ \text{a.e in }\Omega,  \quad     \| \mathcal{B}(f)\|_{W_0^{1,q}(\O;\mathbb{R}^3)}\le C \|f\|_{L^q(\O)}.
        \ea
    \end{proposition}
    
    \begin{proposition}\label{bogov_e}
        Let $\Omega_\varepsilon$ be defined in \eqref{def-O} and $q\in (1, \infty)$. Then there exists a linear operator
    \ba\nonumber
        \mathcal{B}_\varepsilon: L_0^q(\Omega_\varepsilon) \to W_0^{1,q}(\Omega_\varepsilon; \mathbb{R}^3) 
    \ea
        such that for arbitrary $f \in L_0^q(\Omega_\varepsilon)$ there holds
    \ba\nonumber
    \operatorname{div} \mathcal{B}_\varepsilon(f) = f \ \text{a.e. in}  \ \Omega_\varepsilon, \quad        \|\mathcal{B}_\varepsilon(f)\|_{W_0^{1,q}(\Omega_\varepsilon)}   \le C(1 + \varepsilon^{\frac{(3-q)\alpha-3}{q}})\|f\|_{L^q(\Omega_\varepsilon)},
    \ea
    where the constant $C > 0$ is independent of $\varepsilon$.
    \end{proposition}

\subsection{Improved regularity for shear thickening fluids and shear thinning fluids}
    To derive the convergence rates, higher regularity of $\mathbf{u}$ is often needed. Here, we consider the boundary value problem for steady incompressible viscous non-Newtonian flows with $\mathbb{S}$ satisfying \eqref{power-CY} and $\mu_1=0$:
    \begin{equation}\label{problem}
        \begin{cases}
        -{\rm div}\, \mathbb{S}(D\mathbf{u}) +(\mathbf{u}\cdot\nabla)\mathbf{u} +\nabla p=\mathbf{f},& {\rm in}\;\O, \\
        {\rm div} \,\mathbf{u}=0,& {\rm in}\;\O,\\
        \mathbf{u}=0,& {\rm on}\;\partial\O,
        \end{cases}
    \end{equation}
    where $\O \subset \mathbb{R}^3$ is a bounded domain with $C^{2,1}$ boundary, and the external force $\mathbf{f}$ belongs to $L^{\infty}(\O;\R^{3})$. Previous results in \cite{DRW10} and \cite{BGMS12} guarantee the existence of a weak solution $\mathbf{u}\in W_0^{1,r}(\O)$ to problem \eqref{problem}.  Actually, the regularity can be improved. Firstly, we recall the following propositions for the shear-thickening fluids shown in \cite{BKR11}:    
    \begin{proposition}\label{Shear Thickening Flows}
       Assume $\mathbb{S}(\boldsymbol{\xi})$ takes one of the forms in \eqref{power-CY} with $\delta>0,  \  \mu_0>0, \  \mu_{1} = 0$, and $ r \geq 2$. Let $\Omega \subset \mathbb{R}^3$ be a bounded domain with $C^{2,1}$ boundary, and let $f \in L^\infty(\O)$. Then,  any finite energy  weak solution $\mathbf{u} \in W^{1,r}_0(\O)$ of problem \eqref{problem} satisfies 
        \ba\nonumber
            \mathbf{u}\in W^{1,2r+2}(\Omega). 
        \ea
    \end{proposition}

   For the shear thinning fluids with $1<r<2$, similar improved regularity can be shown under certain smallness assumptions on the external force. Crispo and Grisanti  \cite{CG08} first considered the regularity of weak solutions to \eqref{problem} with $\mathbb{S}$ satisfying the power-law (\eqref{power-CY} with $\delta=0$). Guerra \cite{G12} subsequently generalized this result to the Carreau-Yasuda law (\eqref{power-CY} with $\delta>0$). Here, we recall the results presented in \cite{G12}:
    \begin{proposition}\label{shear thinning fluids}
       Assume $\mathbb{S}(\boldsymbol{\xi})$ takes one of the forms in \eqref{power-CY} with $\delta>0, \ \mu_0>0, \ \mu_{1} = 0$, and  $1<r<2$. Let $\O\subset\mathbb{R}^3$ be a bounded domain with $C^{1,\gamma_0}$ boundary with $\gamma_0=1-\frac{3}{q}$ and $q>3$, and let $f \in L^q(\O)$. If $\|\mathbf{f}\|_{L^{q}(\O)}\le\kappa$, where $\kappa$ is a sufficiently small positive constant depending on $q,\gamma_0, r, \mu_0, \delta$ and $\O$, then there exists a unique finite energy weak solution $\uu \in C^{1,\gamma} $ to problem \eqref{problem} for any $\gamma<\gamma_0$ such that 
        \ba\label{r<2regularity}
            \|\mathbf{u}\|_{C^{1,\gamma}(\overline{\O})}\le k_0 \|\mathbf{f}\|_{L^\infty(\O)}.
    \nn    \ea
     \end{proposition}

    \begin{remark}
    \item 
        \begin{itemize}
            \item In \cite{G12}, the authors show that the constant $\kappa$ could be given as  
       $$
                \kappa = \min\big\{1,\frac{1}{d_1(2^{a+4-r}d(a+4-r)^{\frac{1}{a+3-r}})},\frac{1}{k_0(1+k_0)^{2-r}(2-r+2d)}\big\},       \       k_0 = 2d(\frac{a+4-r}{a+3-r}),
           $$
            where $a=a(r, \gamma_{0})$ is a positive real number greater than $2$,  $d$ and $d_1$ are suitable positive constants depending on $\Omega$.
        \end{itemize}
    \end{remark}
    
    \section{Proof of Theorem \ref{thm-1}}
   This section is devoted to proving  the qualitative homogenization result in Theorem \ref{thm-1}. 
   
\subsection{Uniform estimates}\label{Uniform estimates}

By the energy inequality \eqref{energy-eq} and coercivity \eqref{coer}, we have
\ba\label{energy in}
\int_{\Omega}(\delta+|D\tilde{\bf u}_\e|)^{r-2}|D\tilde{\bf u}_\e|^2\,\dx
\leq \int_\Omega\mathbf{f}\cdot \tilde{\bf u}_\e\,{\rm d}x \leq  \|\ff\|_{L^{r'}(\Omega)} \| \tilde{\bf u}_\e\|_{L^r(\Omega)}.
\ea
Firstly, for $r\geq 2$,  using the Poincar\'e inequality gives
\ba\nonumber
\|D\tilde{\bf u}_\e\|_{L^r(\Omega)}^r\leq C\|\mathbf{f}\|_{L^\infty(\O)}\|\nabla \tilde{\bf u}_\e\|_{L^r(\Omega)},
\ea
which, together with Korn inequality (see for example \cite[Chapter 10]{FN09}), yields
\ba\label{uenorm}
\|\tilde{\bf u}_\varepsilon\|_{W^{1,r}_0(\O)}\leq C \|\mathbf{f}\|_{L^\infty(\O)}^{\frac{1}{r-1}}.
\ea

\medskip

Next, consider the case $1<r<2$. Let $\O_{\delta}=\{x\in \Omega : |D\tilde{\bf u}_\e|\geq \delta\}$. Then for $x\in \O_\delta$, there holds
$(\delta+|D\tilde{\bf u}_\e|)^{r-2}\geq (2|D\tilde{\bf u}_\e|)^{r-2}.$
Hence, by \eqref{energy in} and the Poincar\'e inequality we have
\ba\label{dud1}
\int_{\Omega_\delta}|D\tilde{\bf u}_\e|^r\,{\rm d}x\leq C \|\nabla \tilde{\bf u}_\e\|_{L^r(\Omega)}.
\ea
Obviously, 
\ba\label{dux1}
\int_{\Omega\setminus\Omega_\delta}|D\tilde{\bf u}_\e|^r\,{\rm d}x\leq \delta^r|\Omega|.
\ea
Consequently, adding the estimates in \eqref{dud1} and \eqref{dux1} gives
\ba\nonumber
\|D\tilde{\bf u}_\e\|_{L^r(\Omega)}^r\leq C(1+\|\nabla \tilde{\bf u}_\e\|_{L^r(\Omega)}),
\ea
which, together with the Korn inequality, implies
$
\|\tilde{\bf u}_\varepsilon\|_{W^{1,r}_0(\O)}\leq C.
$
Therefore, summarizing both estimates for $r\geq 2$ and for $1<r<2$, we  always have 
\ba\label{est of ue}
\|\tilde{\bf u}_\varepsilon\|_{W^{1,r}_0(\O)}\leq C.
\ea

\medskip

Next we give the estimates of the stress tensor $\mathbb{S}(D \mathbf{u}_\e)$. By virtue of \eqref{growth}, 
\ba\nonumber
|\mathbb{S}(D\tilde{\mathbf{u}}_\e)|\leq C (|D\tilde{\mathbf{u}}_\e|+|D\tilde{\mathbf{u}}_\e|^{r-1})  \ \mbox{for $r\geq 2$}; \quad |\mathbb{S}(D\tilde{\mathbf{u}}_\e)|\leq C |D\tilde{\mathbf{u}}_\e|^{r-1}   \ \mbox{for $r < 2$}.
\ea
It follows from \eqref{est of ue} that
\ba\label{est-S}
\|\mathbb{S}(D\tilde{\mathbf{u}}_\e)\|_{L^{r'}(\O)} \leq C (\|D\tilde{\mathbf{u}}_\e\|_{L^{r}(\O)}+\|D\tilde{\mathbf{u}}_\e\|_{L^{r}(\O)}^{r-1}) \leq C.
\ea

\medskip

Finally, we derive the estimates of the pressure.  The classical theory for incompressible fluids ensures the existence of pressure  $p_{\e} \in L^{q}_{0}(\Omega_{\e})$ for some $q>1$; see \cite[Lemma II.2.2.2]{Sohr01} for details. The key is to derive its uniform bound.  By  estimate \eqref{est of ue} and Sobolev embedding, we have 
$$\|\tilde{\mathbf{u}}_\e\|_{L^{r^*}(\Omega)}\leq C, \quad  r^{*} = \frac{3r}{3-r}.$$

\medskip

    We start with the case $r\geq \frac 95$, which ensures   $2r'\leq r^*$. Then
    \ba\label{2r u}
    \|\tilde{\mathbf{u}}_\e\otimes \tilde \vu_{\e}\|_{L^{r'}(\O)} \leq \|\tilde{\mathbf{u}}_\e\|_{L^{2r'}(\Omega)}^{2} \leq \|\tilde{\mathbf{u}}_\e\|_{L^{r^{*}}}^{2} \leq C.
    \ea
    For arbitrary $\phi\in L^r({\Omega_\e})$, we introduce $\bfvarphi_\e: =\mathcal{B}_\e\big(\phi-\frac{1}{|\O_\e|}\int_{\O_\e} \phi \,{\rm d}x \big)$ which, due to Proposition \ref{bogov_e}, satisfies
    \ba\nonumber
        \|\bfvarphi_\e\|_{W_0^{1,r}(\O_\e)}\le C(1 + \varepsilon^{\frac{(3-r)\alpha-3}{r}})\|\phi\|_{L^r(\O_\e)}.
    \ea
  Using $\bfvarphi_\e$ as the test function yields
    \ba\nonumber
        \langle\nabla p_\e,\bfvarphi_\e\rangle_{\O_\e}=\langle {\rm div}\,\big(\mathbb{S}(D\mathbf{u}_\varepsilon)\big) - (\mathbf{u}_\varepsilon\cdot\nabla)\mathbf{u}_\varepsilon + \mathbf{f} , {\bfvarphi}_\e\rangle_{\O_\e}.
    \ea
Thus, by \eqref{est-S} and  \eqref{2r u}, we obtain
    \ba\nonumber
        \big|\langle\nabla p_\e,\bfvarphi_\e\rangle_{\O_\e}\big|
        & \le
        \|\mathbb{S}(D{\mathbf{u}}_\e)\|_{L^{r'}(\O_\e)}\|\nabla \bfvarphi_\e\|_{L^{r}(\O_\e)}
        \\
        &\quad 
        + \|{\mathbf{u}}_\e\otimes \vu_{\e}\|_{L^{r'}(\O_\e)} \|\nabla \bfvarphi_\e\|_{L^r(\O_\e)}+
        \|\mathbf{f}\|_{L^{r'}(\O_\e)}\|\bfvarphi_\e\|_{L^r(\O_\e)}
        \\
        & \le C \|\bfvarphi_\e\|_{W^{1,r}_0(\O_\e)} \leq    C (1 + \varepsilon^{\frac{(3-r)\alpha-3}{r}}) \|\phi\|_{L^r(\O_\e)}.
    \ea
Hence, by the zero mean property of $p_{\e}$, 
    \ba\nonumber
       \big|\int_{\O_\e} p_\e\,\phi\,{\rm d}x\big | = \big | \int_{\O_\e} p_\e\,{\rm div}\,\bfvarphi_\e\,{\rm d}x\big | =\big|\langle\nabla p_\e,\bfvarphi_\e\rangle_{\O_\e}\big| \le C (1 + \varepsilon^{\frac{(3-r)\alpha-3}{r}}) \|\phi\|_{L^r(\O_\e)}.
    \ea
    This means 
    \ba\label{est-p-r>2}
   \big | p_{\e}\|_{L^{r'}(\Omega_{\e})} \leq C, \quad \mbox{whenever $\frac{9}{5}\leq  r < 3 -\frac{3}{\a}$}.
    \ea

\medskip

    We next consider the case $ \frac{6(\a - 1)}{4\alpha-5}< r<\frac{9}{5}$, for which
    $$
    \frac{1}{r} - \frac{1}{(\frac{r^{*}}{2})'} = \frac{1}{r} - \big(1-\frac{2}{r^{*}}\big) = 3 \big(\frac{1}{r} - \frac{5}{9}\big) >0
    $$
    and
    \ba\label{sigma-star}
    \sigma_{*} := \frac{(3-{(\frac{r^*}{2})'})\alpha-3}{{(\frac{r^*}{2})'}} = \frac{(3-\frac{3r}{5r-6})\alpha-3}{\frac{3r}{5r-6}} = \frac{(4\a - 5) r - 6(\a -1)}{r} >0.
    \ea
     We see the constraint $ r > \frac{6(\a - 1)}{4\alpha-5} $  ensures $\sigma_{*} >0$. For arbitrary $\phi\in L^{(\frac{r^*}{2})'}(\O_\e)$, let
    $
        \bfvarphi_\e: =\mathcal{B}_\e\big(\phi-\frac{1}{|\O_\e|}\int_{\O_\e} \phi \,{\rm d}x \big).
     $
     Thus, due to the fact that $r<(\frac{r^*}{2})'$ whenever $r<\frac{9}{5}$, we have
     \ba\nonumber
        \big|\langle\nabla p_\e,\bfvarphi_\e\rangle_{\O_\e}\big|
        & \le
        \|\mathbb{S}(D{\mathbf{u}}_\e)\|_{L^{r'}(\O_\e)}\|\nabla \bfvarphi_\e\|_{L^{r}(\O_\e)}
        \\
        & \quad        + \|{\mathbf{u}}_\e\otimes \vu_{\e}\|_{L^{\frac{r^*}{2}}(\O_\e)} \|\nabla \bfvarphi_\e\|_{L^{(\frac{r^*}{2})'}(\O_\e)}+
        \|\mathbf{f}\|_{L^{\frac{r^*}{2}}(\O_\e)}\|\bfvarphi_\e\|_{L^{(\frac{r^*}{2})'}(\O_\e)}
        \\
        & \le C \|\bfvarphi_\e\|_{W^{1,{(\frac{r^*}{2})'}}_0(\O_\e)} \le C(1+\e^{\sigma_*})\|\phi\|_{L^{(\frac{r^*}{2})'}(\O_\e)} \leq C \|\phi\|_{L^{(\frac{r^*}{2})'}(\O_\e)} .
    \ea
    Thus, by similar arguments as the derivation of \eqref{est-p-r>2}, we obtain 
    \ba\label{est-p-r<2}
    \|p_\e\|_{L^{\frac{r^*}{2}}_0(\O_\e)}\leq C , \quad \mbox{whenever $\frac{6(\a - 1)}{4\alpha-5}< r < \frac{9}{5}$}.
    \ea 

\medskip

In conclusion, for $\frac{6(\alpha-1)}{4\alpha-5}<r<3-\frac{3}{\alpha}$, we have the following uniform estimates 
\ba\label{sol est}  
\|\tilde{\mathbf{u}}_\e\|_{W_{0}^{1,r}(\Omega)}  + \| \mathbb{S}(D\tilde{\mathbf{u}}_\e) \|_{L^{r'}(\Omega)} + \|\tilde{p}_\e \|_{L^{\min\{r',\frac{r^*}{2}\}}(\Omega)} \leq C.
\ea
Consequently, there hold the following convergences up to possible extractions of subsequences
\ba\label{sol convergence}
&\tilde{\mathbf{u}}_\e\to\mathbf{u} \ {\rm weakly \ in} \ W^{1,r}_0(\O),  \ \tilde{\mathbf{u}}_\e\to\mathbf{u} \ {\rm strongly \ in} \ L^{q_0}(\O) \ \mbox{for any $1\le q_0<r^*$}, \\
& \mathbb{S}(D\tilde{\mathbf{u}}_\e)\to \overline{\mathbb{S}(D\mathbf{u})} \quad {\rm weakly \ in} \ L^{r'}(\O),  \ \tilde{p}_\e\to p \ {\rm weakly \ in} \ L^{\min\{r',\frac{r^*}{2}\}}(\O).
\ea

\subsection{Momentum equations in homogeneous domain}\label{Momentum equations}

To show the limit equations by applying the solenoidal Lipschitz truncation, we need to extend the original equations to the whole domain without holes. The following proposition shows that the extension $(\tilde{\bf u}_\e,\tilde p_\e)$ satisfies the original equations in $\O$ up to a small error.
\begin{proposition}\label{steady extend equation}
    Under the assumptions in Theorem \ref{thm-1}, the zero extension $(\tilde{\bf u}_\e,\tilde p_\e)$ satisfies the following equations in the sense of distribution:
    \ba\label{steady extend estimate}
        -{\rm div}\, \mathbb{S}(D\tilde{\mathbf{u}}_\e) +(\tilde{ \mathbf{u}}_\e\cdot\nabla)\tilde{ \mathbf{u}}_\e+\nabla \tilde p_\e=\mathbf{f}+{\bf G}_\e,
    \ea
    where ${\bf G}_\e\in \mathcal{D}'(\O)$ satisfying
    \ba\label{steady  Ge est}
        |\left\langle{\bf G}_\e,\bfvarphi \right\rangle|\leq C\e^\sigma\|\bfvarphi\|_{W^{1,q}(\O)}, \quad \forall\,\bfvarphi\in C_c^\infty(\O;\mathbb{R}^3),
    \ea
    for some $\sigma>0$ and some $\max\{r, r'\}<q <3-\frac{3}{\alpha}$.
\end{proposition}
\begin{proof}
    By the description of holes in \eqref{hole description},  there exists a family of cut-off functions $\{g_\e\}_{\e>0}$ satisfying 
    \ba\label{def of ge}
  0\leq g_\e\leq1, \   g_\e=0\ {\rm in}\bigcup_{k\in K_\e}B( x_{\e,k},\delta_1a_\e), \
    g_\e=1\ {\rm in} \,\big(\bigcup_{k\in K_\e}B( x_{\e, k},\delta_2a_\e)\big)^c, \ |\nabla g_\e|\leq C\e^{-\alpha}.
\ea
Then for each $1\leq q<\infty$ there holds
\ba\label{g estimate}
\|g_\e-1\|_{L^q(\O)}\leq C\e^{\frac{3(\alpha-1)}{q}},\quad \|\nabla g_\e\|_{L^q(\O)}\leq C\e^{\frac{(3-q)\alpha-3}{q}}.
\ea

Let ${\bfvarphi}\in C_c^\infty(\O;\mathbb{R}^3)$. To extend the original equations from $\O_\e$ to $\O$, we utilize test functions constructed in \eqref{def of ge} and decompose ${\bfvarphi}$ as
\ba\label{test decomposition}
\bfvarphi=g_\e\bfvarphi + (1-g_\e)\bfvarphi.
\ea
Then $g_\e\bfvarphi$ can be chosen as test functions for the original equations \eqref{1.1} in $\O_\e$. For the remaining part $(1-g_\e)\bfvarphi$, we show that the related terms are small.

Now we estimate the following integral
\ba\nonumber
I^\e:=\int_\O\left(\mathbb{S}(D\tilde{\mathbf{u}}_\e):\nabla\bfvarphi+(\tilde{ \mathbf{u}}_\e\cdot\nabla)\tilde{ \mathbf{u}}_\e\cdot\bfvarphi-\tilde{p}_\e{\rm div}\,\bfvarphi-\mathbf{f}\cdot \bfvarphi\right)\,{\rm d}x.
\ea
By the decomposition \eqref{test decomposition} and the weak formulation \eqref{weak-eq-e}, we can rewrite $I^\e$ as 
\ba\nonumber
I^\e&=\int_{\O_\e}\mathbb{S}(D{\mathbf{u}}_\e):\nabla(g_\e\bfvarphi)+({ \mathbf{u}}_\e\cdot\nabla){ \mathbf{u}}_\e\cdot (g_\e\bfvarphi)-p_\e{\rm div}\,(g_\e\bfvarphi)-\mathbf{f}\cdot (g_\e\bfvarphi)\,{\rm d}x+\sum_{j=1}^4 I_j = \sum_{j=1}^4 I_j ,
\ea
where
\ba\nonumber
&I_1=\int_\O\mathbb{S}(D\tilde{\mathbf{u}}_\e):(\nabla\bfvarphi) (1-g_\e)-\mathbb{S}(D\tilde{\mathbf{u}}_\e):(\nabla g_\e\otimes \bfvarphi)\,{\rm d}x,\\
&I_2=\int_\O(\tilde{ \mathbf{u}}_\e\cdot\nabla)\tilde{ \mathbf{u}}_\e\cdot \bfvarphi(1-g_\e)\,{\rm d}x,\\
&I_3=\int_\Omega \tilde{p}_\e{\rm div}\,\big((g_\e-1)\bfvarphi\big)\,{\rm d}x,\\
&I_4=\int_\O\mathbf{f}\cdot \bfvarphi(g_\e-1)\,{\rm d}x .
\ea

 For $I_1$, by \eqref{sol est} and \eqref{g estimate}, together with the assumption $r<3-\frac{3}{\alpha}$, we have 
\ba\label{I1-1}
|I_1| & \leq \|\mathbb{S}(D\tilde{\mathbf{u}}_\e)\|_{L^{r'}(\O)}\big(\|1-g_\e\|_{L^{p_1}(\O)}\|\nabla \bfvarphi\|_{L^{q_1}(\O)} +\|\bfvarphi\|_{L^{p_1}(\O) }\|\nabla g_\e\|_{L^{q_1}(\O)} \big) \\
& \leq  C\e^{\frac{3(\alpha-1)}{p_{1}}}  \|\nabla \bfvarphi\|_{L^{q_1}(\O)} + C  \e^{\frac{(3-q_1)\alpha-3}{q_1}}\| \bfvarphi\|_{L^{p_1}(\O)},
\ea
where we choose $r<q_1< 3 - \frac{3}{\a}<3$ and $r<p_{1}<\infty$ such that
\ba\nonumber
\frac{1}{p_{1}} + \frac{1}{q_{1}} = 1 -  \frac{1}{r'} = \frac1r.
\ea
This ensures the powers of $\e$ in \eqref{I1-1} are positive:
$
\frac{3(\alpha-1)}{p_{1}}>0, \frac{(3-q_1)\alpha-3}{q_1}>0. 
$
Thus, by Sobolev embedding, we have for some $r < q < 3 - \frac{3}{\a}$ and $\sigma > 0$
\ba\nonumber
|I_1|  \leq  C\e^{\sigma}  \| \bfvarphi\|_{W^{1, q}(\O)}.
\ea

\medskip

For $I_{2}$, again by \eqref{sol est} and \eqref{g estimate}, we have 
\ba\nonumber
|I_2| \leq \|\tilde \uu_{\e}\|_{L^{r^{*}}(\Omega)}  \| \nabla \tilde \uu_{\e}\|_{L^{r}(\Omega)}  \|1-g_\e\|_{L^{p_2}(\O)}\| \bfvarphi\|_{L^{q_2}(\O)},
\ea
where $r' < p_{2}, q_{2} < \infty$ satisfy
$
\frac{1}{p_{2}} + \frac{1}{q_{2}}  = 1 - \frac{1}{r^{*}} - \frac{1}{r} = \frac{4}{3} - \frac{2}{r} >0. 
$
By Sobolev embedding, we have for some $\max \{r, r'\} < q <3$ and $\sigma > 0$ that 
\ba\nonumber
|I_2|  \leq  C\e^{\sigma}  \| \bfvarphi\|_{W^{1, q}(\O)}. 
\ea

\medskip

    If $\frac{9}{5}\le r<3-\frac{3}{\alpha}$, the estimate of $I_{3}$ follows exactly from the estimate of $I_{1}$, due to the fact $\tilde p_{\e}$ and $\mathbb{S}(D\tilde{\mathbf{u}}_\e)$ are both {uniformly bounded} in $L^{r'}(\Omega)$.   

    For $\frac{6\alpha-6}{4\alpha-5}<r<\frac{9}{5}$, we have 
    $$
    1-\frac{2}{r^{*}}= \frac{5}{3}-\frac{2}{r}>\frac{\alpha}{3\alpha-3}.
    $$
    Thus, there exist $r<p_3<\infty$ and $r<q_3<3-\frac{3}{\alpha}$ such that 
    \ba\nonumber
        |I_3| & \le \|\tilde{p}_\e\|_{L^{\frac{r^*}{2}}(\O)}\big(\|1-g_\e\|_{L^{p_3}(\O)}\|\nabla \bfvarphi\|_{L^{q_3}(\O)} +\|\bfvarphi\|_{L^{p_3}(\O) }\|\nabla g_\e\|_{L^{q_3}(\O)} \big)
        \\
        & \leq  C\e^{\frac{3(\alpha-1)}{p_{3}}}  \|\nabla \bfvarphi\|_{L^{q_3}(\O)} + C \e^{\frac{(3-q_3)\alpha-3}{q_3}}\| \bfvarphi\|_{L^{p_3}(\O)}.
    \ea
     Then by Sobolev embedding, we can obtain for some $\max\{r,r'\}<q<3$ and $\sigma>0$, 
    \ba\nonumber 
        |I_3|\le C\e^{\sigma}\|\bfvarphi\|_{W^{1, q}(\O)}.
    \ea
    
    \medskip
    
   The estimate of $I_{4}$ is rather straightforward:
   $$
   |I_{4}| \leq C \|{\bf f}\|_{L^{\infty}(\Omega)} \| \bfvarphi\|_{L^{2}(\Omega)} \|1 - g_{\e}\|_{L^{2}(\Omega)} \leq C \| \bfvarphi\|_{L^{2}(\Omega)} \e^{\frac{3(\a-1)}{2}}. 
   $$
   
  The proof is  completed. 
\end{proof}

Therefore,  by Proposition \ref{steady extend equation} and the convergences in \eqref{sol convergence}, passing $\e \to 0$ in \eqref {steady extend estimate} implies
\ba\label{limit equ}
-{\rm div}\, \overline{\mathbb{S}(D\mathbf{u})} + (\mathbf{u}\cdot\nabla)\mathbf{u}+\nabla p=\mathbf{f} \quad {\rm in} \ \mathcal{D'}(\Omega).
\ea
To complete the proof of Theorem \ref{thm-1}, it remains to show the limit of the nonlinear stress tensor
\ba\label{SD-SD}
\overline{\mathbb{S}(D\mathbf{u})}=\mathbb{S}(D\mathbf{u}).
\ea 
This will be done in the following section. 

\subsection{Strong convergence of the stress tensor}\label{Strong convergence of the velocity gradient}

    In this section, we establish the strong convergence of the nonlinear stress tensor, namely $\mathbb{S}(D\tilde{\bf u}_\e) \to \mathbb{S}(D{\bf u})$ a.e. in $\Omega$, and consequently $\overline{\mathbb{S}(D\mathbf{u})}=\mathbb{S}(D\mathbf{u})$. The idea is to subtract the limit equation \eqref{limit equ} from \eqref{steady extend estimate} and then test the resulting difference with the difference of the solutions, $\tilde{\bf u}_\e-{\bf u}$. 
    Formally, this yields
    \ba\nonumber
        \int_\O\big(\mathbb{S}(D\tilde{\mathbf{u}}_\e)-\overline{\mathbb{S}(D{\mathbf{u}})}\big): (D\tilde \vu_{\e} -D \vu) - (\tilde{ \mathbf{u}}_\e\otimes\tilde{ \mathbf{u}}_\e-{\bf u}\otimes{\bf u}\big):(\nabla \tilde \vu_{\e} -\nabla  \vu)  \,{\rm d}x =\langle{\bf G}_\e,\tilde \vu_{\e} - \vu \rangle.
     \ea
   A priori, $\tilde{\bf u}_\e-{\bf u}$ only belongs to $ W^{1,r}_{0}(\Omega;\R^{3})$. However, according to Proposition \ref{steady extend equation}, the test functions should be chosen as $\bfvarphi\in W^{1,q}_0(\Omega;\R^{3})$ with $q>\max \{r, r'\}$, for which \eqref{steady  Ge est} holds. We therefore need to construct smoother test functions that differ only slightly from $\tilde{\bf u}_\e-{\bf u}$. For this purpose, we adopt the idea that a solenoidal function in $W_0^{1,q}$ can be approximated by a solenoidal Lipschitz function that coincides with the original one except on sets of small Lebesgue measure. This is the idea of {\em solenoidal Lipschitz truncations}, which has been employed in various areas of analysis in different contexts, such as the calculus of variations and the existence theory of partial differential equations.
    Next, we state the following proposition, which provides the properties of such solenoidal Lipschitz truncations. See, for example, \cite{BDS013}, \cite{DMS08}, or \cite{FMS03} for details.
    \begin{proposition}\label{lip tru}
        Let $1<q<\infty$, and let ${\bf w}_m\to0$ weakly in $W^{1,q}_0(\O)$ be a solenoidal sequence. Then there exist $k_0\in\mathbb{N}$ and a sequence $\{\lambda_{m, k}\}_{m\in \N, k\geq k_{0}} $ with $2^{2^k}\leq \lambda_{m,k}<2^{2^{k+1}-1}$ such that the Lipschitz truncations $\{{\bf w}^k_m\}$ satisfy the following properties for any $m\in \N$ and any $k\geq k_{0}$ :
        \begin{enumerate}[(i)]
            \item ${\bf w}^k_m\in W^{1,\infty}_0(\mathbb{R}^3)$ with $\operatorname{div}{\bf w}^k_m=0$, and ${\bf w}^k_m={\bf w}_m$ on $\O\setminus E_{m,k}$, where $E_{m,k}$ are measurable  sets satisfying $\|\lambda_{m,k}\chi_{E_{m,k}}\|_{L^q(\mathbb{R}^3)}\le C2^{-\frac{k}{q}}\|\nabla {\bf w}_m\|_{L^q(\O)}$. Here $\chi_{E_{m,k}}$ denotes the characteristic function of the set $E_{m,k}$;
            \item $\|{\bf w}^k_m\|_{W^{1,\infty}(\Omega)}\le C\lambda_{m,k}$;
            \item ${\bf w}^k_m\to0$ as $m\to\infty$ in $L^\infty(\Omega)$.
        \end{enumerate}
    \end{proposition}

    \medskip

    We then apply Proposition \ref{steady extend equation} and Proposition \ref{lip tru} to prove \eqref{SD-SD}. Let $\{\tilde{{\bf u}}_\e-{\bf u}\}_{\e>0}$ be a solenoidal family.  From \eqref{sol convergence}, we know that there exists a subsequence $\{\e_{m}\}_{m\in \N}$ such that $\e_{m} \to 0$,  as $m\to \infty$ and the subsequence $\{\vv_{m}\}_{m\in \N}$ with $\vv_{m} := \tilde{{\bf u}}_{\e_{m}} - {\bf u}$ satisfies 
    \ba\nonumber
    \dive \vv_{m} = 0, \quad     {\bf v}_{m} \to0 \ {\rm weakly \ in }\ W_0^{1,r}(\O) \ \mbox{and strongly in} \ L^{r}(\Omega), \quad  \mbox{as $m\to \infty$}.
    \ea
    By Proposition \ref{lip tru}, there exists  $\{{\bf v}_m^k\}_{m\in \N, k\geq k_{0}} \subset W^{1,\infty}(\O)$  with $\operatorname{div}{\bf v}_m^k=0$ and  $\{\lambda_{m, k}\}_{m\in \N, k\geq k_{0}} $ with $2^{2^k}\leq \lambda_{m,k}<2^{2^{k+1}-1}$  such that for all $m\in \N$ and all $k\geq k_{0}$:
    \ba\label{vk estimate}
        \|{\bf v}_m^k\|_{W^{1,\infty}(\O)}\leq C\lambda_{m,k},\quad {\bf v}_m^k={\bf v}_m \ {\rm in} \ \O\setminus E_{m, k},\quad {\bf v}^k_m \to 0 \ {\rm in} \ L^\infty(\Omega) \ {\rm as} \ m\to \infty .
    \ea
    Moreover,  $\|\lambda_{m,k}\chi_{E_{m,k}}\|_{L^r(\R^{3})}\le C 2^{-\frac{k}{r}}\|\nabla {\bf v}_m\|_{L^r(\O)} \leq C 2^{-\frac{k}{r}}$ implies
    \ba\label{ek measure}
        |E_{m,k} |\leq C\lambda_{m,k}^{-r}2^{-k}.
    \ea

     Subtracting \eqref{limit equ} from \eqref{steady extend estimate}, we obtain
     \ba\label{basic equation}
        \int_\O\big(\mathbb{S}(D\tilde{\mathbf{u}}_{\e_{m}})-\overline{\mathbb{S}(D{\mathbf{u}})}\big):D\bfvarphi-(\tilde{ \mathbf{u}}_{\e_{m}}\otimes\tilde{ \mathbf{u}}_{\e_{m} }- {\bf u}\otimes{\bf u}\big):\nabla\bfvarphi-(\tilde p_{\e_{m}} -p)\,{\rm div}\,\bfvarphi\,{\rm d}x=\langle{\bf G}_{\e_{m}},\bfvarphi\rangle
     \ea
     for any $\bfvarphi\in W^{1,q}_0(\O;\R^{3})$ with $\max \{r, r'\}< q<3$. We then consider a cut-off function $\zeta\in C_c^\infty(\Omega)$ with $0\le\zeta\le1$, choose the test function $\bfvarphi=\zeta{\bf v}_m^k$ in \eqref{basic equation}, and pass to the limit $m\to \infty$ term by term. We shall repeatedly use the uniform bounds and convergences in \eqref{sol convergence}. 
     
     \medskip
     
     For the nonlinear convective term, we use \eqref{vk estimate}  to deduce
     \ba\label{nonlinear term}
        &\big | \int_{\O}(\tilde{ \mathbf{u}}_{\e_{m}} \otimes\tilde{ \mathbf{u}}_{\e_{m}}  -{\bf u}\otimes{\bf u}):\nabla(\zeta{\bf v}_m^k)\,{\rm d}x \big|
        \\
        & \le \big | \int_{\O}(\tilde{ \mathbf{u}}_{\e_{m}} \otimes\tilde{ \mathbf{u}}_{\e_{m}} -{\bf u}\otimes{\bf u}):\zeta\nabla{\bf v}_m^k \,{\rm d}x \big | +\big |  \int_{\O}(\tilde{ \mathbf{u}}_{\e_{m}} \otimes\tilde{ \mathbf{u}}_{\e_{m}} -{\bf u}\otimes{\bf u}):{\bf v}_m^k\otimes\nabla\zeta\,{\rm d}x \big | 
        \\
        & \le  C \lambda_{m,k} \|\tilde{ \mathbf{u}}_{\e_{m}} \otimes\tilde{ \mathbf{u}}_{\e_{m}} - {\bf u}\otimes{\bf u}\|_{L^{1}(\Omega)} \\
        & \leq C  \lambda_{m,k}   \|\tilde{ \mathbf{u}}_{\e_{m}} -  {\bf u}\|_{L^{2}(\Omega)} .
     \ea
     
     As ${\rm div}\,{\bf v}_m^k=0$,  we have
    \ba\nonumber
       \big | \int_\Omega(\tilde p_{\e_{m}} - p)\,{\rm div}\,(\zeta{\bf v}_m^k)\,{\rm d}x\big |  =
        \big | \int_\Omega(\tilde p_{\e_{m}} - p)\,\nabla\zeta\cdot{\bf v}_m^k\,{\rm d}x\big | \le C \|\tilde p_{\e_{m}}-p\|_{L^{\min\{r',\frac{r^*}{2}\}}(\Omega)}\|{\bf v}^k_m\|_{L^\infty(\O)}.
    \ea
    
    By \eqref{steady  Ge est} we have
    \ba\label{ref estimates}
        |\langle {\bf G}_{\e_{m}},\zeta{\bf v}_m^k\rangle| \leq C\e_{m}^\sigma\|\zeta\mathbf{v}_m^k\|_{W^{1,q}(\Omega)}\leq C\lambda_{m,k}\e_{m}^\sigma.
    \ea
    Therefore, from \eqref{basic equation}--\eqref{ref estimates} we can deduce that
    \ba\label{vek-0}
         \big | \int_\O \big(\mathbb{S}(D\tilde{\mathbf{u}}_{\e_{m}})-\overline{\mathbb{S}(D{\mathbf{u}})}\big):D(\zeta{\bf v}_m^k)\,{\rm d}x \big |    \le  C\big(\lambda_{m,k} \e_{m}^\sigma+\|{\bf v}^k_m\|_{L^\infty(\O)}
        + \lambda_{m,k}  \|\tilde{ \mathbf{u}}_{\e_{m}} -  {\bf u}\|_{L^{2}(\Omega)}    \big).
    \ea

    \medskip
    
    We then compute 
     \ba\nonumber
        &   \int_{\O\setminus E_{m,k}}\big(\mathbb{S}(D\tilde{\mathbf{u}}_{\e_{m}}) -\overline{\mathbb{S}(D{\mathbf{u}})}\big):\zeta D  \vv_{m} \,{\rm d}x   \\
        & \quad = \int_{\O\setminus E_{m,k}}\big(\mathbb{S}(D\tilde{\mathbf{u}}_{\e_{m}}) -\overline{\mathbb{S}(D{\mathbf{u}})}\big):\zeta D  \vv_{m}^{k} \,{\rm d}x  \\
        & \quad =  \int_{\O\setminus E_{m,k}}\big(\mathbb{S}(D\tilde{\mathbf{u}}_{\e_{m}}) -\overline{\mathbb{S}(D{\mathbf{u}})}\big): D(\zeta  \vv_{m}^{k} ) \,{\rm d}x  -   \int_{\O\setminus E_{m,k}}\big(\mathbb{S}(D\tilde{\mathbf{u}}_{\e_{m}}) -\overline{\mathbb{S}(D{\mathbf{u}})}\big):\big(\nabla\zeta\otimes \vv_{m}^{k} \big)\,{\rm d}x \\
        & \quad = \int_{\O }\big(\mathbb{S}(D\tilde{\mathbf{u}}_{\e_{m}}) -\overline{\mathbb{S}(D{\mathbf{u}})}\big): D(\zeta  \vv_{m}^{k} ) \,{\rm d}x  -   J_{1} - J_{2}  , 
            \ea
where
    \ba\label{nable zeta}
    J_{1}  &: = \int_{\O\setminus E_{m,k}}\big(\mathbb{S}(D\tilde{\mathbf{u}}_{\e_{m}}) -\overline{\mathbb{S}(D{\mathbf{u}})}\big):\big(\nabla\zeta\otimes \vv_{m}^{k} \big)\,{\rm d}x  \leq C  \| \vv_{m}^{k} \|_{L^{r}(\O)},\\
     J_{2}   &: =   \int_{\Omega \cap E_{m,k}}\big(\mathbb{S}(D\tilde{\mathbf{u}}_{\e_{m}})-\overline{\mathbb{S}(D{\mathbf{u}})}\big):D(\zeta{\bf v}_m^k)\,{\rm d}x \\
    & \quad  \leq  C \| \mathbb{S}(D\tilde{\mathbf{u}}_{\e_{m}})-\overline{\mathbb{S}(D{\mathbf{u}})}\|_{L^{r'}(\Omega)} \| \vv_{m}^{k}\|_{W^{1,\infty}(\R^{3})}   \|1\|_{L^{r}(E_{m,k})} \\
   &  \quad \leq C\lambda_{m,k} |E_{m,k}|^{\frac{1}{r}}  \\
   & \quad  \leq C2^{-\frac{k}{r}}.
    \ea
    Therefore, by \eqref{vek-0}--\eqref{nable zeta}, we deduce that for each $k\in\mathbb{N}$
    \ba\nonumber
        & \big | \int_{\O\setminus E_{m,k}}\big(\mathbb{S}(D\tilde{\mathbf{u}}_{\e_{m}}) -\overline{\mathbb{S}(D{\mathbf{u}})}\big):\zeta D (\tilde{\bf u}_{\e_{m}} - {\bf u} )\,{\rm d}x\big | 
        \\
        & \quad \le C\big(\lambda_{m,k}\e_{m}^\sigma+\|{\bf v}^k_m\|_{L^\infty(\O)}+2^{-\frac{k}{r}} + \lambda_{m,k} \|\tilde{ \mathbf{u}}_{\e_{m}} - {\bf u} \|_{(L^{2}\cap L^{r})(\Omega)}
        \big).
    \ea
    Thus, as $m\to \infty$ ($\e_{m}\to0$), by virtue of the strong convergence $\tilde \vu_{\e_{m}} \to \vu$ in $L^{q_{0}}(\Omega)$ for any $q_{0} < r^{*}$ and ${\bf v}^k_m\to0$ in $L^\infty(\O)$, we have
    \ba\nonumber
        \limsup_{m \to 0}\big | \int_{\O\setminus E_{m,k}}\big(\mathbb{S}(D\tilde{\mathbf{u}}_{\e_{m}}) -\overline{\mathbb{S}(D{\mathbf{u}})}\big):\zeta D(\tilde{\bf u}_{\e_{m}}-{\bf u})\,{\rm d}x\big | \le C 2^{-\frac{k}{3}}.
    \ea
    Then, there exists a subsequence $ \{m_{k}\}_{k\geq k_{0}}$  such that
    \ba\label{ek-1}
       \big | \int_{\O\setminus E_{m_{k},k}}\big(\mathbb{S}(D\tilde{\mathbf{u}}_{\e_{m_{k}}})-\overline{\mathbb{S}(D{\mathbf{u}})}\big):\zeta D(\tilde{\bf u}_{\e_{m_{k}}}-{\bf u})\,{\rm d}x\big | & \leq C2^{-\frac{k}{4}}, \ \forall \, k\geq k_{0}.
    \ea
    
  Let $\displaystyle \zeta_k:=\zeta \chi_{\O\setminus E_{m_{k},k}}  $. We have 
    \ba\nonumber
        \zeta_k(x) \to \zeta(x) \ {as}\ k\to\infty , \quad \forall x\in  \liminf_{k\to \infty} (\O\setminus E_{m_{k},k}) = \Omega \setminus \limsup_{k \to \infty}   E_{m_{k},k} .
    \ea
    As $\displaystyle |E_{m_{k},k}|\leq C\lambda_{m,k}^{-r} 2^{-k}\le C 2^{-r 2^{k}} 2^{-k}$, we have  $|\limsup_{k\to\infty}E_{m_{k},k} |=0$ and thus
    \ba\label{a.e.convergence}
        \zeta_k(x)\to\zeta(x) \ {as}\ k\to\infty,  \ a.e.  \ x \in \Omega.
    \ea
    
    By \eqref{ek-1}, we have
    \ba\label{ek-2}
        \lim_{k\to\infty} \big | \int_{\O}\big(\mathbb{S}(D\tilde{\mathbf{u}}_{\e_{m_{k}}})-\overline{\mathbb{S}(D{\mathbf{u}})}\big):\zeta_k D(\tilde{\bf u}_{\e_{m_{k}}}-{\bf u})\,{\rm d}x  \big| =0.
    \ea
    Combining \eqref{sol convergence} and \eqref{a.e.convergence}, we can deduce from \eqref{ek-2} that
    \ba\nonumber
        \lim_{k\to\infty}\int_{\O}\mathbb{S}(D\tilde{\mathbf{u}}_{\e_{m_{k}}}):\zeta_k D\tilde{\bf u}_{\e_{m_{k}}}\,{\rm d}x=\int_{\O}\overline{\mathbb{S}(D{\mathbf{u}})}:\zeta D{\bf u}\,{\rm d}x.
    \ea
Applying the local Minty's trick which we recall in {Appendix \ref{Minty Trick}} (see Proposition \ref{Minty Trick}) implies
    \ba\nonumber
        \zeta\overline{\mathbb{S}(D{\mathbf{u}})}=\zeta\mathbb{S}(D{\mathbf{u}}) \ {a.e.}\ {\rm in} \ \O.
    \ea
    Taking the limit $\zeta\to\chi_{\O}$  we finally derive that $\displaystyle\overline{\mathbb{S}(D{\mathbf{u}})}=\mathbb{S}(D{\mathbf{u}}) \ {a.e.}\ {\rm in} \ \O$. We thus complete the proof of Theorem \ref{thm-1}. 

    \section{Proof of Theorem \ref{convergence thm} }\label{Convergence Rate}
    
    This section is devoted to the proof of the quantitative homogenization results in Theorem \ref{convergence thm} with  $\mathbb{S}(\boldsymbol{\xi})$ satisfying \eqref{power-CY} with $\mu_{0}>0, \mu_{1}=0, \de>0$ and with the smallness condition on external force $\|{\bf f}\|_{L^{\infty}(\Omega)} \leq \kappa\ll 1$.
    
    \subsection{Convergence rates for velocity}
    We now consider the convergence rate for the velocity field $\tilde{\mathbf{u}}_\epsilon - \mathbf{u}$. We compute
    \ba\nonumber
        \int_{\O} & \big(\mathbb{S}(D\tilde{\mathbf{u}}_{\e})-{\mathbb{S}(D{\mathbf{u}})}\big):D(\tilde{\bf u}_{\e}-{\bf u})\,{\rm d}x 
        \\
        & = \int_{\O}\big(\mathbb{S}(D\tilde{\mathbf{u}}_{\e}):D\tilde{\bf u}_{\e}+{\mathbb{S}(D{\mathbf{u}})}:D{\bf u}-\mathbb{S}(D\tilde{\mathbf{u}}_{\e}):D{\bf u}-{\mathbb{S}(D{\mathbf{u}})}:D\tilde{\bf u}_{\e}\big)\,{\rm d}x.
    \ea
    Choosing ${\mathbf{u}}$ and $\tilde{\bf u}_{\e}$ as test functions in the limit equation \eqref{limit equ} yields
    \ba\nonumber
        & \int_{\O}{\mathbb{S}(D{\mathbf{u}})}:D{\bf u}\,{\rm d}x=\int_{\O}{\bf f}\cdot{\bf u} \,{\rm d}x,\\
       &  \int_{\O}{\mathbb{S}(D{\mathbf{u}})}:D\tilde{\bf u}_{\e}\,{\rm d}x=
        \int_{\O}\left(-{\bf u}\cdot\nabla{\bf u}\cdot \tilde{\bf u}_{\e}+{\bf f}\cdot\tilde{\bf u}_{\e}\right)\,{\rm d}x.
    \ea
    Similarly, choosing $g_\epsilon \mathbf{u}$ as a test function in \eqref{1.1} yields
    \ba\nonumber
        \int_{\O}\mathbb{S}(D\tilde{\mathbf{u}}_{\e}):D(g_\e{\bf u})\,{\rm d}x =
        \int_{\O_\e}\left(-{\bf u}_{\e}\cdot\nabla{\bf u}_{\e}\cdot (g_\e{\bf u})-{p}_{\e}\operatorname{div}(g_\e{\bf u})+{\bf f}\cdot (g_\e{\bf u})\right)\,{\rm d}x.
    \ea
  For the nonlinear convective terms,  a direct calculation yields
    \ba\nonumber
        &\int_{\O} ({\bf u}\cdot\nabla{\bf u} )\cdot \tilde{\bf u}_{\e} + (\tilde{\bf u}_{\e}\cdot\nabla\tilde{\bf u}_{\e})\cdot{\bf u}\,{\rm d}x
        \\
        & \quad = \int_{\O} ({\bf u}\cdot\nabla{\bf u} )\cdot (\tilde{\bf u}_{\e}-{\bf u})   + \big({\bf u}_{\e}\cdot\nabla(\tilde{\bf u}_{\e}-{\bf u}) \big)\cdot{\bf u}\,{\rm d}x
          \\
         & \quad  =  -\int_{\O} (\tilde{\bf u}_{\e}-{\bf u})\otimes(\tilde{\bf u}_{\e}-{\bf u}):\nabla {\bf u}\,{\rm d}x.
    \ea
    Together with the energy inequality \eqref{energy-eq} we have 
    \ba\label{Sue-Su-1}
        \int_{\O} & \big(\mathbb{S}(D\tilde{\mathbf{u}}_{\e})-{\mathbb{S}(D{\mathbf{u}})}\big):D(\tilde{\bf u}_{\e}-{\bf u})\,{\rm d}x 
        \\
        & \le \int_{\O} \mathbb{S}(D\tilde{\bf u}_\e):D{\bf u}(g_\e-1) + {\bf f}(1-g_\e){\bf u}+\tilde{\bf u}_{\e}\cdot\nabla\tilde{\bf u}_{\e}\cdot(g_\e-1){\bf u}\,{\rm d}x 
        \\
        & \quad + \int_{\O} \tilde{p}_\e \nabla g_\e\cdot{\bf u} - \mathbb{S}(D\tilde{\bf u}_\e): ({\bf u}\otimes\nabla g_\e)\,{\rm d}x-\int_{\O} (\tilde{\bf u}_{\e}-{\bf u})\otimes(\tilde{\bf u}_{\e}-{\bf u}):\nabla {\bf u}\,{\rm d}x
        \\
        & = I_1+I_2+I_3.
    \ea

    \begin{itemize}
        \item {\bf Case  $2\leq r\leq 3 - \frac{3}{\a}$}. Applying Proposition \ref{Shear Thickening Flows} gives $\mathbf{u} \in W^{1,2r+2}(\Omega)$. Hence by Sobolev embedding we have $\mathbf{u}\in L^\infty(\O)$. Furthermore, by weak lower semi-continuity applied to \eqref{uenorm}, it follows that
        \ba\nonumber
            \|\mathbf{u}\|_{W_0^{1,r}(\O)}\le  C\|\mathbf{f}\|_{L^\infty(\O)}^{\frac{1}{r-1}}.
        \ea 
     Applying inequality \eqref{r>2coverity} gives  that for some positive constant  $\sigma_{2}$,
        \ba\label{r>2ineq}
            \sigma_2 ( \|D\tilde{\bf u}_{\e}-D{\bf u}\|^2_{L^2({\O})} +   \|D\tilde{\bf u}_{\e}-D{\bf u}\|^r_{L^r({\O})} ) \le I_1+I_2+I_3.
        \ea
      

        Now we compute $I_1,I_2$ and $I_3$ term by term. For $I_1$, using the estimates in \eqref{g estimate}  gives 
        \ba\nonumber
            |I_1| & \le \|\mathbb{S}(D\tilde{\bf u}_\e)\|_{L^{r'}(\O)}\|D{\bf u}\|_{L^{2r+2}(\O)}\|g_\e-1\|_{L^{\frac{2r(r+1)}{r+2}}(\O)}
            \\
            & \quad +\|{\bf f}\|_{L^{\infty}(\O)}\|{\bf u}\|_{L^{\infty}(\O)}\|g_\e-1\|_{L^1(\O)}+\|\tilde{\bf u}_{\e}\|_{L^{r^*}(\O)}\|\nabla\tilde{\bf u}_{\e}\|_{L^r(\O)}\|{\bf u}\|_{L^{\infty}(\O)}\|g_\e-1\|_{L^{\frac{3r}{4r-6}}(\O)}
            \\
            & \le C\|{\bf u}\|_{W^{1,2r+2}(\O)} \|g_\e-1\|_{L^{\frac{2r(r+1)}{r+2}}(\O)} 
            \\
            & \le C \|{\bf u}\|_{W^{1,2r+2}(\O)}\e^{\frac{3(\a-1)(r+2)}{2r(r+1)}}.
        \ea
        For $I_2$, by  similar arguments, we have
        \ba\nonumber
            |I_2| & \le \|\tilde{p}_\e\|_{L^{r'}(\O)}\|{\bf u}\|_{L^{\infty}(\O)}\|\nabla g_\e\|_{L^{r}(\O)}+ \|\mathbb{S}(D\tilde{\bf u}_\e)\|_{L^{r'}(\O)}\|{\bf u}\|_{L^{\infty}(\O)}\|\nabla g_\e\|_{L^{r}(\O)}
            \\
            & \le C \|{\bf u}\|_{W^{1,2r+2}(\O)}\e^{\frac{(3-r)\a-3}{r}}.
        \ea
        For $I_3$, applying the Poincar\'e and Korn inequalities yields
        \ba\nonumber
            |I_3|\le \|\tilde{\bf u}_{\e}-{\bf u}\|^2_{L^6(\O)}\|\nabla {\bf u}\|_{L^{\frac{3}{2}}(\O)}\le C\|D\tilde{\bf u}_{\e}-D{\bf u}\|^2_{L^2({\O})}\|\nabla {\bf u}\|_{L^{\frac{3}{2}}(\O)}.
        \ea
        In conclusion, we can deduce from \eqref{r>2ineq} that
        \ba\nonumber
             \sigma_2 & (\|D\tilde{\bf u}_{\e}-D{\bf u}\|^2_{L^2({\O})} + \|D\tilde{\bf u}_{\e}-D{\bf u}\|^r_{L^r({\O})} )
             \\
             & \le C\|D\tilde{\bf u}_{\e}-D{\bf u}\|^2_{L^2({\O})}\|\nabla {\bf u}\|_{L^{\frac{3}{2}}(\O)} + C\|{\bf u}\|_{W^{1.2r+2}(\O)}[\e^{\frac{3(\a-1)(r+2)}{2r(r+1)}}+\e^{\frac{(3-r)\a-3}{r}}].
        \ea
    Choosing $\kappa$ sufficiently small so that $C\|\nabla{\bf u}\|_{L^\frac{3}{2}(\Omega)}\le C\|\nabla{\bf u}\|_{L^r(\Omega)} \leq C\|\mathbf{f}\|_{L^\infty(\O)}^{\frac{1}{r-1}} <\sigma_2$ gives
        \ba\label{rate-v-f}
            \|\tilde{\bf u}_{\e}-{\bf u}\|^2_{W^{1,2}({\O})}+\|\tilde{\bf u}_{\e}-{\bf u}\|^r_{W^{1,r}({\O})} \le C\|{\bf u}\|_{W^{1,2r+2}(\O)}\e^{\frac{(3-r)\alpha-3}{r}}.
        \ea
    \item {\bf Case $\frac{9}{5}\leq r< 2$}. By Proposition \ref{shear thinning fluids}, there holds $\mathbf{u}\in W^{1,\infty}(\O)$ and $\|\mathbf{u}\|_{W^{1,\infty}(\O)}\le C\|\mathbf{f}\|_{L^\infty(\O)}$. Then by \eqref{Sue-Su-1}, together with \eqref{est of ue}, which means $\|D\tilde{\mathbf{u}}_\e\|_{L^r(\O)} + \|D\mathbf{u}\|_{L^r(\O)}\le C$, applying \eqref{r<2coverity} and \eqref{r<2mono} gives
    \ba\nonumber
        \sigma_5 \|D\tilde{\bf u}_{\e}-D{\bf u}\|^2_{L^r({\O})}\le I_1+I_2+I_3,
    \ea
    where the constant $\sigma_5: = \mu_0(r-1)2^{\frac{r-2}{2}}\big(\delta^{2-r} |\Omega|^{\frac{2-r}{r}}+\|D\tilde{\mathbf{u}}_\e\|^{2-r}_{L^r(\O)}+\|D\mathbf{u}\|^{2-r}_{L^r(\O)}\big)^{-1} > 0.$

    Now we compute $I_1,I_2$ and $I_3$ term by term. For $I_1$, direct computation gives
    \ba\nonumber
        |I_1| & \le \|\mathbb{S}(D\tilde{\bf u}_\e)\|_{L^{r'}(\O)}\|D{\bf u}\|_{L^{\infty}(\O)}\|g_\e-1\|_{L^{r}(\O)}
        \\
        & \quad +\|{\bf f}\|_{L^{\infty}(\O)}\|{\bf u}\|_{L^{\infty}(\O)}\|g_\e-1\|_{L^1(\O)}+\|\tilde{\bf u}_{\e}\|_{L^{r^*}(\O)}\|\nabla\tilde{\bf u}_{\e}\|_{L^r(\O)}\|{\bf u}\|_{L^{\infty}(\O)}\|g_\e-1\|_{L^{\frac{3r}{4r-6}}(\O)}
        \\
        & \le C\|{\bf u}\|_{W^{1,\infty}(\O)} \|g_\e-1\|_{L^{\frac{3r}{4r-6}}(\O)} 
        \\
        & \le C \|{\bf u}\|_{W^{1,\infty}(\O)}\e^{\frac{(\a-1)(4r-6)}{r}}.
    \ea
    For $I_2$, there holds
    \ba\nonumber
        |I_2| & \le \|\tilde{p}_\e\|_{L^{r'}(\O)}\|{\bf u}\|_{L^{\infty}(\O)}\|\nabla g_\e\|_{L^{r}(\O)}+ \|\mathbb{S}(D\tilde{\bf u}_\e)\|_{L^{r'}(\O)}\|{\bf u}\|_{L^{\infty}(\O)}\|\nabla g_\e\|_{L^{r}(\O)}
        \\
        & \le C \|{\bf u}\|_{W^{1,\infty}(\O)}\e^{\frac{(3-r)\a-3}{r}}.
    \ea
    For $I_3$, applying the Poincar\'e and Korn inequalities yields
    \ba\nonumber
        |I_3|\le \|\tilde{\bf u}_{\e}-{\bf u}\|^2_{L^{r^*}(\O)}\|\nabla {\bf u}\|_{L^{\frac{3r}{5r-6}}(\O)}\le C\|D\tilde{\bf u}_{\e}-D{\bf u}\|^2_{L^r({\O})}\|\nabla {\bf u}\|_{L^{\infty}(\O)}.
    \ea
    In conclusion,
    \ba\nonumber
        \sigma_5 \|D\tilde{\bf u}_{\e}-D{\bf u}\|^2_{L^r({\O})} \le C\|D\tilde{\bf u}_{\e}-D{\bf u}\|^2_{L^r({\O})}\|\nabla {\bf u}\|_{L^{\infty}(\O)} + C \|{\bf u}\|_{W^{1,\infty}(\O)}[\e^{\frac{(3-r)\a-3}{r}}+\e^{\frac{(\a-1)(4r-6)}{r}}].
    \ea
    Choosing $\kappa$ sufficiently small such that $C \|\mathbf{u}\|_{W^{1,\infty}(\O)}\le C\|\mathbf{f}\|_{L^{\infty}(\O)} < \sigma_{5}$ implies 
    \ba\label{9/5<r<2-vrate}
        \|\tilde{\bf u}_{\e}-{\bf u}\|^2_{W^{1,r}({\O})}\le C\|{\bf u}\|_{W^{1,\infty}(\O)}\e^{\frac{(3-r)\alpha-3}{r}}.
    \ea
    \item {\bf Case  $\frac{6(\a-1)}{4\alpha-5}<r<\frac{9}{5}$}. Similar to the case $\frac{9}{5}\leq r<2$, we have that 
    \ba\nonumber
        \sigma_5 \|D\tilde{\bf u}_{\e}-D{\bf u}\|^2_{L^r({\O})}\le I_1+I_2+I_3,
    \ea
with
    \ba\nonumber
        |I_1|   \le C \|{\bf u}\|_{W^{1,\infty}(\O)}\e^{\frac{(\a-1)(4r-6)}{r}}, \quad 
        |I_3|   \le C\|D\tilde{\bf u}_{\e}-D{\bf u}\|^2_{L^r({\O})}\|\nabla {\bf u}\|_{L^{\infty}(\O)}.
    \ea
The estimates for $I_2$ is different: 
    \ba\nonumber
        |I_2| & \le \|\tilde{p}_\e\|_{L^{\frac{r^*}{2}}(\O)}\|{\bf u}\|_{L^{\infty}(\O)}\|\nabla g_\e\|_{L^{(\frac{r^*}{2})'}(\O)}+ \|\mathbb{S}(D\tilde{\bf u}_\e)\|_{L^{r'}(\O)}\|{\bf u}\|_{L^{\infty}(\O)}\|\nabla g_\e\|_{L^{r}(\O)}
        \\
        & \le C \|{\bf u}\|_{L^{\infty}(\O)}\big(\e^{\frac{(4\a-5)r-6(\a-1)}{r}}+\e^{\frac{(3-r)\a-3}{r}}\big)
        \\
        & \le C \|{\bf u}\|_{L^{\infty}(\O)}\e^{\frac{(4\a-5)r-6(\a-1)}{r}}.
    \ea
    Hence, 
    \ba\nonumber
        \sigma_5 \|D\tilde{\bf u}_{\e}-D{\bf u}\|^2_{L^r({\O})} 
        & \le C\|D\tilde{\bf u}_{\e}-D{\bf u}\|^2_{L^r({\O})}\|\nabla {\bf u}\|_{L^{\infty}(\O)} 
        \\
        & \qquad + C \|{\bf u}\|_{W^{1,\infty}(\O)}[\e^{\frac{(4\a-5)r-6(\a-1)}{r}}+\e^{\frac{(\a-1)(4r-6)}{r}}].
    \ea
  Choosing $\kappa$ sufficiently small such that $C\|\mathbf{u}\|_{W^{1,\infty}(\O)}\le C\|\mathbf{f}\|_{L^{\infty}(\O)} < \sigma_5$ implies
    \ba\label{r<9/5-vrate}
        \|\tilde{\bf u}_{\e}-{\bf u}\|^2_{W^{1,r}({\O})}\le C\|{\bf u}\|_{W^{1,\infty}(\O)}\e^{\frac{(4\a-5)r-6(\a-1)}{r}}.
    \ea
    \end{itemize}
    This completes the proof of the convergence rates for velocity field in $\eqref{rate-0r>2}$, \eqref{rate-095<r<2} and \eqref{rate-0r<95}.
    
\subsection{Convergence rates for pressure}

    Now we consider the convergence rate of $\tilde{p}_\e-p$. Concerning the estimate of the pressure $p$, the classical theory of distributions ensures that $\|p\|_{L^{q}(\Omega)} \leq \|\nabla p\|_{W^{-1,q}(\O)}$ for any $p\in L^{q}_{0}(\Omega)$ with $q\in (1, \infty)$; see, for example, \cite[Lemma II.2.2.2]{Sohr01}. Here, we briefly prove the improved integrability of the pressure $p$, since ${\bf u}$ has improved regularity by Proposition \ref{Shear Thickening Flows} and Proposition \ref{shear thinning fluids}.
    
    For $r\ge2$, Proposition \ref{Shear Thickening Flows} yields $\mathbf{u}\in W^{1,2r+2}(\O)$, and hence $\vu \in L^{\infty}(\Omega)$ and $\mathbb{S}(D\mathbf{u})\in L^{\frac{2r+2}{r-1}}(\O)$. Thus, for any $\phi\in C^\infty_c(\Omega)$, by employing the Bogovskii operator $\mathcal{B}$ introduced in Proposition \ref{bogov}, we have
    \ba\label{imp-p-1}
        |\langle p,\phi \rangle_{\O}| & = |\langle \nabla p, \mathcal{B}(\phi-\frac{1}{|\O|}\int_\O \phi\,{\rm d}x)\rangle_{\O}|
        \\
        & = |\langle\operatorname{div}(\mathbb{S}(D{\bf u}))-\operatorname{div}({\bf u}\otimes{\bf u})+{\bf f},\mathcal{B}(\phi-\frac{1}{|\O|}\int_\O \phi\,{\rm d}x)\rangle_{\O}|
        \\
        & \leq \big(\|\mathbb{S}(D\mathbf{u})\|_{L^{\frac{2r+2}{r-1}}(\O)}+\|\mathbf{u}\|^2_{L^\infty(\O)}\big)\|\nabla\mathcal{B}(\phi-\frac{1}{|\O|}\int_\O \phi\,{\rm d}x)\|_{L^{\frac{2r+2}{r+3}}(\O)}
        \\
        & \quad + \|\mathbf{f}\|_{L^\infty(\O)}\|\mathcal{B}(\phi-\frac{1}{|\O|}\int_\O \phi\,{\rm d}x)\|_{L^1(\O)}
        \\
        & \leq C \|\phi\|_{L^{\frac{2r+2}{r+3}}(\O)}.
    \ea
This means 
    \ba\label{r>2-p-regularity}
        \|p\|_{L^{\frac{2r+2}{r-1}}(\O)}\leq C.
    \ea
    
    Similarly, for $\frac{6(\a-1)}{4\a-5}<r<2$, by Proposition \ref{shear thinning fluids}, we have ${\bf u}\in W^{1,\infty}(\O)$; hence, for any $1<q<\infty$, we deduce that $\operatorname{div}(\mathbb{S}(D{\bf u}))$ and $\operatorname{div}({\bf u}\otimes{\bf u})$ both belong to $W^{-1,q}(\O)$, which implies that $\nabla p\in W^{-1,q}(\O).$ By similar arguments as in \eqref{imp-p-1},  we have
    \ba\label{r<2-p-regularity}
        \|p\|_{L^q(\O)}\leq C_{q}, \quad \mbox{for any $1<q<\infty$}.
    \ea

    Using the Bogovskii operator $\mathcal{B}_\e$  in Proposition \ref{bogov_e}, we define:
    \ba\nonumber
        \bfvarphi_\e=\mathcal{B}_\e\big(\phi\chi_{\O_\e}-\frac{1}{|\O_\e|}\int_{\O_\e} \phi \,{\rm d}x\big),
    \ea
     where $\chi_{\O_\e}$ denotes the  characteristic function of domain $\Omega_{\e}$.
    Due to the fact that $p_{\e}$ is of mean zero on $\Omega_{\e}$ and $p$ is of mean zero on $\Omega$,  we have
        \ba\nonumber
        |\langle\tilde{p}_\e-p,\phi\rangle_\Omega|
        &=|\langle\tilde{p}_\e-p,\phi\chi_{\O_\e}\rangle_\Omega+\langle\tilde{p}_\e-p,\phi-\phi\chi_{\O_\e}\rangle_\O|
        \\
        &=|\langle\tilde{p}_\e-p,\phi \chi_{\O_\e} -\frac{1}{|\O_\e|}\int_{\O_\e} \phi \,{\rm d} x \rangle_{\Omega_\e}-\langle p,\phi + \frac{1}{|\O_\e|}\int_{\O_\e} \phi \,{\rm d}x\rangle_{\O\setminus\O_\e}|
        \\
        &\le|\langle\nabla(\tilde {p}_\e-p),\bfvarphi_\e\rangle_{\O_\e}|+|\langle p,\phi +  \frac{1}{|\O_\e|}\int_{\O_\e} \phi \,{\rm d}x\rangle_{\O\setminus\O_\e}|
        \\
        &=I_1+I_2.
    \ea

    For $I_1$, using the weak formulations of \eqref{1.1} and \eqref{limit fluid equation} gives
    \ba\nonumber
        I_1 & =|\langle\nabla(\tilde{p}_\e-p),\bfvarphi_\e\rangle_{\O_\e}|
        \\
        & = |\langle\operatorname{div}\big(\mathbb{S}(D\tilde{\mathbf{u}}_\varepsilon)\big)-\operatorname{div}\big(\mathbb{S}(D\mathbf{u}) \big)
        -(\tilde{\mathbf{u}}_\varepsilon\cdot\nabla)\tilde{\mathbf{u}}_\varepsilon+(\mathbf{u}\cdot\nabla)\mathbf{u}, \bfvarphi_\e\rangle_{\O_\e}|
        \\
        &\le|\langle\mathbb{S}(D\tilde{\mathbf{u}}_\e)-\mathbb{S}(D\mathbf{u}),\nabla \bfvarphi_\e\rangle_{\O_\e}|+
        |\langle\tilde{\mathbf{u}}_\e\otimes\tilde{\mathbf{u}}_\e-\mathbf{u}\otimes\mathbf{u},\nabla \bfvarphi_\e\rangle_{\O_\e}|
        \\
        &=I_{11}+I_{12}.
    \ea

    \begin{itemize}
        \item {\bf Case $2\le r<3-\frac{3}{\alpha}$}.  For $I_{11}$, applying the inequality \eqref{eq:continuity_r_greater_2}  yields
    \ba\nonumber
        I_{11}
        & \le \|\mathbb{S}(D\tilde{\mathbf{u}}_\e)-\mathbb{S}(D\mathbf{u})\|_{L^{r'}(\O_\e)}\|\nabla\bfvarphi_\e\|_{L^{r}(\O_\e)}
        \\
        & \leq C\|(\delta^{2} + |D\tilde{\mathbf{u}}_\e|^{2} + |D\mathbf{u}|^{2})^{\frac{r-2}{2}} |D\tilde{\mathbf{u}}_\e - D\mathbf{u}|\|_{L^{r'}(\O_\e)}\|\nabla\bfvarphi_\e\|_{L^{r}(\O_\e)}
        \\
        & \le C\|(\delta^{2} + |D\tilde{\mathbf{u}}_\e|^{2} + |D\mathbf{u}|^{2})^{\frac{r-2}{2}} \|_{L^{\frac{r}{r-2}}(\O_\e)}\|D\tilde{\mathbf{u}}_\e - D\mathbf{u}\|_{L^r(\O_\e)}\|\nabla\bfvarphi_\e\|_{L^{r}(\O_\e)}
        \\
        &\le C\|D\tilde{\mathbf{u}}_\e - D\mathbf{u}\|_{L^r(\O_\e)}\|\nabla\bfvarphi_\e\|_{L^{r}(\O_\e)}
        \\
        &\le C\|\tilde{\mathbf{u}}_\e-\mathbf{u}\|_{W^{1,r}(\O)}\|\nabla\bfvarphi_\e\|_{L^{r}(\O_\e)}.
    \ea
    
    For $I_{12}$, since $2r' \leq r^{*}$ for any $r\ge2$, there holds
    \ba\nonumber
        I_{12}
        & \le
        \|\tilde{\mathbf{u}}_\e\otimes\tilde{\mathbf{u}}_\e-\mathbf{u}\otimes\mathbf{u}\|_{L^{r'}(\O_\e)}\|\nabla \bfvarphi_\e\|_{L^{r}(\O_\e)}
        \\
        &\le 
        C(\|\tilde{\mathbf{u}}_\e\|_{L^{r^*}(\O_\e)}+\|\mathbf{u}\|_{L^{r^*}(\O_\e)})\|\tilde{\mathbf{u}}_\e-\mathbf{u}\|_{L^{r^*}(\O_\e)}\|\nabla \bfvarphi_\e\|_{L^{r}(\O_\e)}
        \\
        & \le 
        C \|\tilde{\mathbf{u}}_\e-\mathbf{u}\|_{W^{1,r}(\O)}\|\nabla \bfvarphi_\e\|_{L^{r}(\O_\e)}.
    \ea
    Therefore,  by the convergence rate  \eqref{rate-0r>2} and the estimate of $\mathcal{B}_{\e}$   in Proposition \ref{bogov_e},  we have 
        \ba\label{I_1estimate}
        I_1  \le C \|\tilde{\mathbf{u}}_\e-\mathbf{u}\|_{W^{1,r}(\O)}\|\nabla \bfvarphi_\e\|_{L^{r}(\O_\e)} \le C \e^{\frac{(3-r)\a-3}{r^2}}\|\phi\|_{L^{r}(\O)}.
    \ea
    For $I_2$, from \eqref{r>2-p-regularity} we deduce that
    \ba\label{I_2-est-1}
        I_2  & = |\langle p,\phi +  \frac{1}{|\O_\e|}\int_{\O_\e} \phi \,{\rm d}x  \rangle_{\O\setminus\O_\e}|  \\
        & \leq \|p\|_{L^{\frac{2r+2}{r-1}}(\Omega)}\|\phi\|_{L^r(\O)}|\O\setminus\O_\e|^{\frac{(r+2)(r-1)}{2r (r+1)}}  +  \|p\|_{L^{\frac{2r+2}{r-1}}(\Omega)}\|\phi\|_{L^1(\O)}  |\O\setminus\O_\e|^{\frac{r+3}{2r+2}} |\Omega_{\e}|^{-1} \\
        & \le C \e^{\frac{3(\a-1)(r+2)(r-1)}{2r(r+1)}}\|\phi\|_{L^r(\O)}.
    \ea
    Thus, by \eqref{rate-v-f}, \eqref{I_1estimate} and \eqref{I_2-est-1}, together with the fact   $\frac{(3-r)\a-3}{r^2} \leq \frac{3(\a-1)(r+2)(r-1)}{2r(r+1)} $,  we obtain
    \ba\nonumber
        |\langle\tilde{p}_\e-p,\phi\rangle_\Omega|  \le I_1 +I_2 
         \le C \e^{\frac{(3-r)\a-3}{r^2}}\|\phi\|_{L^r(\O)}.
    \ea
    
    Hence, for $2\le r<3-\frac{3}{\alpha}$, there holds
    \ba\nonumber
        \|\tilde{p}_\e-p\|_{L^{r'}(\O)}\le C\e^{\frac{(3-r)\alpha-3}{r^2}},
    \ea
    which is exactly our desired convergence rate for pressure in $\eqref{rate-0r>2}$. 
    \item {\bf Case $\frac{9}{5}\le r<2$}. The inequality \eqref{eq:continuity_r_less_2} yields
    \ba\nonumber
        I_{11}
        & \le \|\mathbb{S}(D\tilde{\mathbf{u}}_\e)-\mathbb{S}(D\mathbf{u})\|_{L^{r'}(\O_\e)}\|\nabla\bfvarphi_\e\|_{L^{r}(\O_\e)}
        \\
        & \le C \|D\tilde{\mathbf{u}}_\e - D\mathbf{u}\|^{r-1}_{L^r(\O_\e)}\|\nabla\bfvarphi_\e\|_{L^{r}(\O_\e)}
        \\
        & \le C\|\tilde{\mathbf{u}}_\e-\mathbf{u}\|^{r-1}_{W^{1,r}(\O)}\|\nabla\bfvarphi_\e\|_{L^{r}(\O_\e)}.
    \ea
    For $\frac{9}{5}\le r<2$, there holds $\frac{2}{r^*}+\frac{1}{r}\le 1$ and consequently 
    \ba\nonumber
        I_{12}
        & \le C \|\tilde{\mathbf{u}}_\e\otimes\tilde{\mathbf{u}}_\e-\mathbf{u}\otimes\mathbf{u}\|_{L^{\frac{r^*}{2}}(\O_\e)}\|\nabla \bfvarphi_\e\|_{L^{r}(\O_\e)}
        \\
        & \le C\big(\|{\bf u}\|_{L^{r^*}(\O_\e)}+\|\tilde{\bf u}_\e\|_{L^{r^*}(\O_\e)}\big)\|\tilde{\bf u}_\e-{\bf u}\|_{L^{r^*}(\O_\e)}\|\nabla \bfvarphi_\e\|_{L^{r}(\O_\e)}
        \\
        & \le C \|\tilde{\mathbf{u}}_\e-\mathbf{u}\|_{W^{1,r}(\O)}\|\nabla \bfvarphi_\e\|_{L^{r}(\O_\e)}.
    \ea
    Recall that $p\in L^{q}(\Omega)$ for any $1<q<\infty$ (see \eqref{r<2-p-regularity}). Thus
    \ba\nonumber
        I_2  \le 
        \|p\|_{L^q(\O)}\|\phi\chi_{\O\setminus{\O_\e}}\|_{L^{q'}(\O)}
        \le C \|\phi\|_{L^{r}(\O)}|\O\setminus\O_\e|^{\frac{1}{q'}-\frac{1}{r}}
        \le C \e^{3(\a-1)(\frac{1}{q'}-\frac{1}{r})}\|\phi\|_{L^{r}(\O)}.
    \ea
    Thus, by \eqref{9/5<r<2-vrate} we obtain
    \ba\nonumber
        |\langle\tilde{p}_\e-p,\phi\rangle_\Omega|  \le I_1 +I_2  \le C\big[ \e^{\frac{[(3-r)\a-3](r-1)}{2r}}\big(1+\e^{\frac{(3-r)\a-3}{r}}\big) + \e^{3(\a-1)(\frac{1}{q'}-\frac{1}{r})} \big]\|\phi\|_{L^{r}(\O)}.
    \ea
    We may choose $q$ suitably large such that $3(\a-1)(\frac{1}{q'}-\frac{1}{r}) \geq \frac{[(3-r)\a-3](r-1)}{2r}$. Then
    \ba\nonumber
        |\langle\tilde{p}_\e-p,\phi\rangle_\Omega| \le C\e^{\frac{[(3-r)\a-3](r-1)}{2r}}\big(1+\e^{\frac{(3-r)\a-3}{r}}\big)\|\phi\|_{L^{r}(\O)} \leq C\e^{\frac{[(3-r)\a-3](r-1)}{2r}}\|\phi\|_{L^{r}(\O)} .
    \ea
    This means
    \ba\nonumber
        \|\tilde{p}_\e-p\|_{L^{r'}(\O)}\le C\e^{\frac{[(3-r)\a-3](r-1)}{2r}}.
    \ea
    This gives  our desired convergence rate for pressure in $\eqref{rate-095<r<2}$. 
    
    \item {\bf Case $\frac{6(\a-1)}{4\a-5}<r<\frac{9}{5}$}.
    As the case $\frac{9}{5}\le r<2$,  there holds for any $1<q<\infty$ that
    \ba\nonumber
        I_{11} & \le \|\mathbb{S}(D\tilde{\mathbf{u}}_\e)-\mathbb{S}(D\mathbf{u})\|_{L^{r'}(\O_\e)} \|\nabla\bfvarphi_\e\|_{L^r(\O_\e)}\le C\|\tilde{\mathbf{u}}_\e-\mathbf{u}\|^{r-1}_{W^{1,r}(\O)}\|\nabla\bfvarphi_\e\|_{L^{(\frac{r^*}{2})'}(\O_\e)},
        \\
        I_{12} & \le C \|\tilde{\mathbf{u}}_\e-\mathbf{u}\|_{W^{1,r}(\O)}\|\nabla \bfvarphi_\e\|_{L^{(\frac{r^*}{2})'}(\O_\e)},
        \\
        I_2 & \le C \e^{3(\a-1)(\frac{1}{q'}-\frac{1}{(\frac{r^*}{2})'})} \|\phi\|_{L^{(\frac{r^*}{2})'}(\O)}.
    \ea
   By \eqref{r<9/5-vrate}, choosing $q$ suitably large such that $ 3(\a-1)(\frac{2}{r^*}-\frac{1}{q}) \geq \frac{(4\a-5)r-6(\a-1)}{2r}$, we obtain
    \ba\nonumber
        |\langle\tilde{p}_\e-p,\phi\rangle_\Omega| & \le C\e^{\frac{[(4\a-5)r-6(\a-1)](r-1)}{2r}}\big(1+\e^{\sigma_{*}}\big)\|\phi\|_{L^{(\frac{r^*}{2})'}(\O)} 
        \\
        & \leq C\e^{\frac{[(4\a-5)r-6(\a-1)](r-1)}{2r}}\|\phi\|_{L^{(\frac{r^*}{2})'}(\O)},
    \ea
    where $\sigma_{*}>0$ is given in \eqref{sigma-star}.
    In conclusion, we have
    \ba\nonumber
        \|\tilde{p}_\e-p\|_{L^{\frac{r^*}{2}}(\O)}\le C\e^{\frac{[(4\a-5)r-6(\a-1)](r-1)}{2r}}.
    \ea
    This gives our desired convergence rate for pressure in  \eqref{rate-0r<95}.

    \end{itemize}

\section*{Acknowledgements}
All authors are partially supported by the Natural Science Foundation of Jiangsu Province (Grant No. BK20240058). Y. L. is partially supported by the Fundamental Research Funds for the Central Universities (Grant No. 2026300394).


\newpage
    \begin{appendices}
    \setcounter{table}{0}
    \setcounter{figure}{0}
    \setcounter{equation}{0}
    \renewcommand{\thetable}{\thesection.\arabic{table}}
    \renewcommand{\theequation}{\thesection.\arabic{equation}}
    \section{Additional properties of the stress tensor}
    For the sake of completeness, we present the following properties of the stress tensor satisfying \eqref{power-CY} (see \cite[Chapter 5]{Malek1996} for example). 
    \begin{lemma}\label{potential of the tensor}
    Let $\mathbb{S}(\boldsymbol{\xi})$ be one of the forms in \eqref{power-CY}  with $\delta>0,\mu_0>0$ and $\mu_1=0$. 
    \begin{itemize}
        \item If $r\ge2$,  then  there exist $\sigma_1>0,\sigma_2>0$ such that for all $\boldsymbol{\xi}, \boldsymbol{\zeta} \in \mathbb{M}^3_{sym}$, there holds
        \ba\label{eq:continuity_r_greater_2}
            |\mathbb{S}(\boldsymbol{\xi}) - \mathbb{S}(\boldsymbol{\zeta})| \leq \sigma_{1}  (\delta^2 + |\boldsymbol{\xi}|^2 + |\boldsymbol{\zeta}|^2)^{\frac{r-2}{2}} |\boldsymbol{\xi} - \boldsymbol{\zeta}|,
        \ea
        \ba\label{r>2coverity}
            \big(\mathbb{S}(\boldsymbol{\xi})-\mathbb{S}(\boldsymbol{\zeta})\big):\big(\boldsymbol{\xi}-\boldsymbol{\zeta}\big)
            \ge \sigma_2 ( |\boldsymbol{\xi}-\boldsymbol{\zeta}|^2  + |\boldsymbol{\xi}-\boldsymbol{\zeta}|^r).
        \ea
     
        \item If $1<r<2$, then there exists $\sigma_3 >0, \sigma_{4}>0$,  such that  for all $\boldsymbol{\xi}, \boldsymbol{\zeta} \in \mathbb{M}^3_{sym}$, there holds
        \ba\label{eq:continuity_r_less_2}
            \big|\mathbb{S}(\boldsymbol{\xi}) - \mathbb{S}(\boldsymbol{\zeta})\big| \le \sigma_3|\boldsymbol{\xi} - \boldsymbol{\zeta}|^{r-1}, 
        \ea
          \ba\label{r<2coverity}
            (\mathbb{S}(\boldsymbol{\xi})-\mathbb{S}(\boldsymbol{\zeta})):(\boldsymbol{\xi}-\boldsymbol{\zeta})
            \ge 
            \sigma_4(\delta+|\boldsymbol{\xi}|^2+|\boldsymbol{\zeta}|^2)^{\frac{r-2}{2}}|\boldsymbol{\xi}-\boldsymbol{\zeta}|^2.
        \ea
     Moreover, for any $\boldsymbol{\xi}, \boldsymbol{\zeta} \in L^r(\O)$,  
                \ba\label{r<2mono}
          \int_{\Omega}(\mathbb{S}(\boldsymbol{\xi})-\mathbb{S}(\boldsymbol{\zeta})):(\boldsymbol{\xi}-\boldsymbol{\zeta})\, \dx
            \ge
            \frac{\sigma_4\|\boldsymbol{\xi}-\boldsymbol{\zeta}\|^2_{L^r(\O)}}{\delta^{2-r} |\Omega|^{\frac{2-r}{r}}+\|\boldsymbol{\xi}\|^{2-r}_{L^r(\O)}+\|\boldsymbol{\zeta}\|^{2-r}_{L^r(\O)}}.
        \ea
    \end{itemize}
    
    \end{lemma}
    
    \begin{proof} Here we only give the proof for the typical Carreau-Yasuda law case, i.e., $\mathbb{S}(\boldsymbol{\xi}) = \mu_0(\delta^2 + |\boldsymbol{\xi}|^2)^{\frac{r-2}{2}}\boldsymbol{\xi}$. The other case can be proved similarly.

\medskip

Define $\boldsymbol{\eta}(t) := \boldsymbol{\zeta} + t(\boldsymbol{\xi} - \boldsymbol{\zeta})$  and compute
    \begin{equation}\label{S-xi-eta}
        |\mathbb{S}(\boldsymbol{\xi}) - \mathbb{S}(\boldsymbol{\zeta})| = \big|\int_0^1 \frac{d}{dt}\mathbb{S}(\boldsymbol{\eta}(t)) dt\big| = \big|\int_0^1 D\mathbb{S}(\boldsymbol{\eta}(t)) : (\boldsymbol{\xi} - \boldsymbol{\zeta}) dt\big|,
    \end{equation}
    where $D\mathbb{S}$ denotes the Jacobian matrix of $\mathbb{S}$ satisfying for any symmetric matrix $\boldsymbol{V}$ that
    \begin{equation}\label{DV-inequality}
        D\mathbb{S}(\boldsymbol{\eta}) : \boldsymbol{V} = \mu_0 \big( (\delta^2 + |\boldsymbol{\eta}|^2)^{\frac{r-2}{2}}\boldsymbol{V} + (r-2)(\delta^2 + |\boldsymbol{\eta}|^2)^{\frac{r-4}{2}}(\boldsymbol{\eta} : \boldsymbol{V})\boldsymbol{\eta} \big),
    \end{equation}
and
    \ba\nonumber
        D\mathbb{S}(\boldsymbol{\eta}) : (\boldsymbol{V}\otimes \boldsymbol{V}) 
        = 
        \mu_0\big( (\delta^2 + |\boldsymbol{\eta}|^2)^{\frac{r-2}{2}}|\boldsymbol{V}|^2 + (r-2)(\delta^2 + |\boldsymbol{\eta}|^2)^{\frac{r-4}{2}}(\boldsymbol{\eta} : \boldsymbol{V})^2 \big).
    \ea
    
    \medskip
    
 We first consider the case $r\geq 2 $ and prove \eqref{r>2coverity} and \eqref{eq:continuity_r_greater_2}. With $r\geq 2$, we have 
    \ba\nonumber
        D\mathbb{S}(\boldsymbol{\eta}) : (\boldsymbol{V}\otimes \boldsymbol{V})
        \ge
        \mu_0(\delta^2 + |\boldsymbol{\eta}|^2)^{\frac{r-2}{2}}|\boldsymbol{V}|^2.
    \ea
    Taking $\boldsymbol{V}=\boldsymbol{\xi} - \boldsymbol{\zeta}$ and integrating over $t\in[0,1]$, we get 
    \ba\label{S-r>2}
        (\mathbb{S}(\boldsymbol{\xi})  -\mathbb{S}(\boldsymbol{\zeta})):(\boldsymbol{\xi}-\boldsymbol{\zeta})
        \ge
        \mu_0(\frac{1}{2})^{\frac{r-2}{2}}|\boldsymbol{\xi}-\boldsymbol{\zeta}|^2\int_0^1 \big(\delta^2 +|\boldsymbol{\zeta}+s(\boldsymbol{\xi}-\boldsymbol{\zeta})|\big)^{r-2}{\rm d}s.
    \ea
  By considering the cases $|\boldsymbol{\zeta}|\ge|\boldsymbol{\xi}-\boldsymbol{\zeta}|$ and $|\boldsymbol{\xi}-\boldsymbol{\zeta}|\ge|\boldsymbol{\zeta}|$ separately, we can conclude from \eqref{S-r>2} that
    \ba\nonumber
        \big(\mathbb{S}(\boldsymbol{\xi})-\mathbb{S}(\boldsymbol{\zeta})\big):\big(\boldsymbol{\xi}-\boldsymbol{\zeta}\big)
        \ge
        \sigma_2 ( |\boldsymbol{\xi}-\boldsymbol{\zeta}|^2 + |\boldsymbol{\xi}-\boldsymbol{\zeta}|^r),
     \ea
    where 
    \ba\label{CY-law}   
        \sigma_{2}= \mu_0 \big (\frac{1}{2} \big )^{\frac{r}{2}} \min\{ \delta^{r-2} ,   \frac{1}{r-1} , \frac{1}{ 12^{\frac{r}{2}}} \}.
    \ea
    
   We turn to prove  \eqref{eq:continuity_r_greater_2}. With  $r \geq  2$,  we have
      \ba\nonumber
        |D\mathbb{S}(\boldsymbol{\eta}) : \boldsymbol{V}| \leq \mu_0 \big( (\delta^2 + |\boldsymbol{\eta}|^2)^{\frac{r-2}{2}}|\boldsymbol{V}| + (r-2)(\delta^2 + |\boldsymbol{\eta}|^2)^{\frac{r-4}{2}}|\boldsymbol{\eta}|^2 |\boldsymbol{V}| \big)   = \mu_0 (r - 1)(\delta^2 + |\boldsymbol{\eta}|^2)^{\frac{r-2}{2}}|\boldsymbol{V}|.
    \ea
   Thus, we can deduce from \eqref{S-xi-eta} that
    \begin{equation}\nonumber
        |\mathbb{S}(\boldsymbol{\xi}) - \mathbb{S}(\boldsymbol{\zeta})| \leq \mu_0 (r - 1) |\boldsymbol{\xi} - \boldsymbol{\zeta}| \int_0^1 (\delta^2 + |\boldsymbol{\eta}(t)|^2)^{\frac{r-2}{2}} \, \dt. 
    \end{equation}
    Since $(\delta^2 + |\boldsymbol{\eta}(t)|^2)^{\frac{r-2}{2}} \leq 2^{\frac{r-2}{2}} (\delta^2 + |\boldsymbol{\xi}|^2 + |\boldsymbol{\zeta}|^2)^{\frac{r-2}{2}}$, we have
    \begin{equation*}
        |\mathbb{S}(\boldsymbol{\xi}) - \mathbb{S}(\boldsymbol{\zeta})| \leq \sigma_{1} (\delta^2 + |\boldsymbol{\xi}|^2 + |\boldsymbol{\zeta}|^2)^{\frac{r-2}{2}} |\boldsymbol{\xi} - \boldsymbol{\zeta}|, \quad \sigma_{1}   = \mu_0 (r - 1) 2^{\frac{r-2}{2}}.
    \end{equation*}
   
     \medskip

   We then consider the case $1<r<2$.  By symmetry,  we may assume $|\boldsymbol{\xi}| \le |\boldsymbol{\zeta}|$ and the case $|\boldsymbol{\xi}|\ge |\boldsymbol{\zeta}|$ can be proved similarly. If $|\boldsymbol{\xi}|  \leq  |\boldsymbol{\zeta}|\le |\boldsymbol{\xi} - \boldsymbol{\zeta}|$, with $1<r<2$ we can easily deduce that
    \ba\nonumber
        |\mathbb{S}(\boldsymbol{\xi}) - \mathbb{S}(\boldsymbol{\zeta})| & \le |\mathbb{S}(\boldsymbol{\xi})|+|\mathbb{S}(\boldsymbol{\zeta})|
        \\ 
        & \le \mu_0(\delta^2+|\boldsymbol{\xi}|^2)^{\frac{r-2}{2}}|\boldsymbol{\xi}|+\mu_0(\delta^2+|\boldsymbol{\zeta}|^2)^{\frac{r-2}{2}}|\boldsymbol{\zeta}|
        \\
        & \le 2 \mu_0|\boldsymbol{\xi} - \boldsymbol{\zeta}|^{r-1}
    \ea
    And if $|\boldsymbol{\xi} - \boldsymbol{\zeta}| \le |\boldsymbol{\zeta}|$, there holds $|\boldsymbol{\eta}(t)|\ge(1-t)|\boldsymbol{\xi} - \boldsymbol{\zeta}|$. Then, with $1<r<2$, using \eqref{DV-inequality} gives
        \ba\nonumber
        |D\mathbb{S}(\boldsymbol{\eta}) : \boldsymbol{V}| &\leq \mu_0 \big( (\delta^2 + |\boldsymbol{\eta}|^2)^{\frac{r-2}{2}}|\boldsymbol{V}| + |r-2|(\delta^2 + |\boldsymbol{\eta}|^2)^{\frac{r-4}{2}}|\boldsymbol{\eta}|^2 |\boldsymbol{V}| \big) \\
        &\leq \mu_0 (3 - r )(\delta^2 + |\boldsymbol{\eta}|^2)^{\frac{r-2}{2}}|\boldsymbol{V}|,
    \ea
 and consequently, by \eqref{S-xi-eta}, we have that
    \ba\nonumber
        |\mathbb{S}(\boldsymbol{\xi}) - \mathbb{S}(\boldsymbol{\zeta})| 
        & \leq \mu_0 (3 - r)|\boldsymbol{\xi} - \boldsymbol{\zeta}| \int_0^1 (\delta^2 + |\boldsymbol{\eta}(t)|^2)^{\frac{r-2}{2}} \dt \\ 
        & \leq \mu_0 (3 - r)|\boldsymbol{\xi} - \boldsymbol{\zeta}| \int_0^1 (\delta^2 + (1-t)^{2} | \boldsymbol{\xi} - \boldsymbol{\zeta}  |^2)^{\frac{r-2}{2}} \dt \\ 
    &  \leq \mu_0\frac{3-r}{r-1}|\boldsymbol{\xi} - \boldsymbol{\zeta}|^{r-1}. 
    \ea
We thus derive our desired result \eqref{eq:continuity_r_less_2} with $\sigma_3=\max\{2\mu_0,\mu_0\frac{3-r}{r-1}\}$.
    
    \medskip
    
    For  $1<r<2$, we proceed the same way as in the proof of \eqref{eq:continuity_r_greater_2} to  obtain
        \begin{equation}\nonumber
        (\mathbb{S}(\boldsymbol{\xi}) - \mathbb{S}(\boldsymbol{\zeta})) : (\boldsymbol{\xi} - \boldsymbol{\zeta}) \ge (r-1) |\boldsymbol{\xi} - \boldsymbol{\zeta}|^2 \int_0^1 (\delta^2 + |\boldsymbol{\eta}(t)|^2)^{\frac{r-2}{2}} \, {\rm d}t. 
        \end{equation}
      Using the fact $|\boldsymbol{\eta}(t)|^2 \le 2(|\boldsymbol{\xi}|^2 + |\boldsymbol{\zeta}|^2)$, we can thus get 
        \begin{equation}\nonumber
        (\mathbb{S}(\boldsymbol{\xi}) - \mathbb{S}(\boldsymbol{\zeta})) : (\boldsymbol{\xi} - \boldsymbol{\zeta}) \ge (r-1)2^{\frac{r-2}{2}}(\delta^2 + |\boldsymbol{\xi}|^2 + |\boldsymbol{\zeta}|^2)^{\frac{r-2}{2}} |\boldsymbol{\xi} - \boldsymbol{\zeta}|^2,
        \end{equation}
     which is our desired result \eqref{r<2coverity} with $\sigma_{4} = (r-1)2^{\frac{r-2}{2}}$.

        \medskip
        
        For \eqref{r<2mono}, integrating the pointwise inequality \eqref{r<2coverity} over domain $\Omega$ implies
        \begin{equation}\nonumber
        \int_{\Omega} \big( \mathbb{S}(\boldsymbol{\xi}) - \mathbb{S}(\boldsymbol{\zeta}) \big) : (\boldsymbol{\xi} - \boldsymbol{\zeta}) \, {\rm d}x \ge \sigma_{4}\int_{\Omega} (\delta^2 + |\boldsymbol{\xi}|^2 + |\boldsymbol{\zeta}|^2)^{\frac{r-2}{2}} |\boldsymbol{\xi} - \boldsymbol{\zeta}|^2 \, {\rm d}x.
        \end{equation}
        Applying H\"older's inequality   yields
        \ba\nonumber
        \int_{\Omega} |\boldsymbol{\xi} - \boldsymbol{\zeta}|^r \, {\rm d}x 
        & = \int_{\Omega} \big( |\boldsymbol{\xi} - \boldsymbol{\zeta}|^r (\delta^2 + |\boldsymbol{\xi}|^2 + |\boldsymbol{\zeta}|^2)^{\frac{r(r-2)}{4}} \big) \big( (\delta^2 + |\boldsymbol{\xi}|^2 + |\boldsymbol{\zeta}|^2)^{\frac{r(2-r)}{4}} \big) \, {\rm d}x
        \\
        & \le \big( \int_{\Omega} |\boldsymbol{\xi} - \boldsymbol{\zeta}|^2 (\delta^2 + |\boldsymbol{\xi}|^2 + |\boldsymbol{\zeta}|^2)^{\frac{r-2}{2}} \, {\rm d}x \big )^{\frac{r}{2}} \big ( \int_{\Omega} (\delta^2 + |\boldsymbol{\xi}|^2 + |\boldsymbol{\zeta}|^2)^{\frac{r}{2}} \, {\rm d}x \big )^{\frac{2-r}{2}}.
        \ea
        Direct computation gives
        \begin{equation}\nonumber
        \big ( \int_{\Omega} (\delta^2 + |\boldsymbol{\xi}|^2 + |\boldsymbol{\zeta}|^2)^{\frac{r}{2}} \, {\rm d}x \big )^{\frac{2-r}{r}} \le \delta^{2-r}|\Omega|^{\frac{2-r}{r}} + \|\boldsymbol{\xi}\|_{L^r(\Omega)}^{2-r} + \|\boldsymbol{\zeta}\|_{L^r(\Omega)}^{2-r}.
        \end{equation}
        Thus we have,
        \begin{equation}\nonumber
        \int_{\Omega} \big( \mathbb{S}(\boldsymbol{\xi}) - \mathbb{S}(\boldsymbol{\zeta}) \big) : (\boldsymbol{\xi} - \boldsymbol{\zeta}) \, {\rm d}x \ge \frac{ \sigma_{4}\|\boldsymbol{\xi} - \boldsymbol{\zeta}\|_{L^r(\Omega)}^2}{\delta^{2-r} |\Omega|^{\frac{2-r}{r}} + \|\boldsymbol{\xi}\|_{L^r(\Omega)}^{2-r} + \|\boldsymbol{\zeta}\|_{L^r(\Omega)}^{2-r}}.
        \end{equation}
    \end{proof}

    \section{Localization of the Minty trick}
    For the completeness and the convenience of  readers, we give a brief proof of Minty's trick, following \cite{Wolf2007}.
    \begin{lemma}\label{Minty Trick}
         Assume the stress tensor $\mathbb{S}:\mathbb{M}^d_{sym}\to \mathbb{M}^d_{sym}$ satisfies the {\em $r$-structure} \eqref{growth}$-$\eqref{monotonicity}. Let $\{\mathbf{u}_n\}_{ n\in\mathbb{N} }\subset W^{1,r}(\O;\mathbb{R}^d)$ and $\{\phi_n\}_{ n\in\mathbb{N} }\subset  L^\infty(\O)$ are given sequences, such that 
         \ba\nonumber
             0\le\phi_n\le a_0 \ a.e. \ in\ \O, \ \mbox{for some constant $a_{0}$},
         \ea
         \ba\nonumber
             D\mathbf{u}_n\to D\mathbf{u} \ weakly\ in \ L^r(\O),
         \ea
         \ba\nonumber
             \mathbb{S}(D\mathbf{u}_n)\to \overline{\mathbb{S}(D{\mathbf{u}})}\ weakly\ in \ L^{r'}(\O),
         \ea
         \ba\nonumber
             \phi_n\to\phi \ a.e. \ in \ \O \ as \ n\to\infty.
         \ea
  In addition suppose that  
         \ba\label{limit prop}
             \limsup_{n\to\infty}\int_\O \mathbb{S} (D\mathbf{u_n} ) : D\mathbf{u_n} \phi_n \,{\rm d}x 
             = \int_\O \overline{\mathbb{S}(D{\mathbf{u}})}:D\mathbf{u}\phi \,{\rm d}x.
         \ea
         Then 
         \ba\nonumber
             \overline{\mathbb{S}(D{\mathbf{u}})}\phi=\mathbb{S}(D\mathbf {u})\phi  \ a.e.\ in \ \O.
         \ea
    \end{lemma}     
    \begin{proof}
    As for the usual Minty's trick, it suffices to show that for any $\mathbf{v} \in W^{1,r}(\Omega;\mathbb{R}^n)$,
    \begin{equation} \label{eq:target}
        \int_\Omega (\overline{\mathbb{S}(D\mathbf{u})}-\mathbb{S}(D\mathbf{v})):(D\mathbf{u}-D\mathbf{v}) \phi \, {\rm d}x\ge0.
    \end{equation}
    Starting from the monotonicity condition $$\int_\Omega (\mathbb{S}(D\mathbf{u}_n)-\mathbb{S}(D\mathbf{v})):(D\mathbf{u}_n-D\mathbf{v}) \phi_n \, {\rm d}x\ge0,$$ we decompose the integral into three parts:
    \ba\nonumber
        A_n = \int_\Omega \mathbb{S}(D\mathbf{u}_n):D\mathbf{u}_n \phi_n \, {\rm d}x, \quad
        B_n = \int_\Omega \mathbb{S}(D\mathbf{u}_n):D\mathbf{v} \phi_n \, {\rm d}x, \quad
        C_n = \int_\Omega \mathbb{S}(D\mathbf{v}):(D\mathbf{u}_n - D\mathbf{v}) \phi_n \, {\rm d}x.
    \ea
    By assumption \eqref{limit prop}, $\limsup_{n\to \infty} A_n = \int_\Omega \overline{\mathbb{S}(D\mathbf{u})}:D\mathbf{u} \phi \, {\rm d}x$.
    For $B_n$ and $C_n$, the weak convergence of $\mathbb{S}(D\mathbf{u}_n)$ and $D\mathbf{u}_n$, combined with the strong convergence of $\phi_n \to \phi$ (due to boundedness and a.e. convergence via Egorov's theorem or dominated convergence theorem), implies
    \ba\nonumber
        \lim_{n\to\infty} (B_n + C_n) = \int_\Omega \overline{\mathbb{S}(D\mathbf{u})}:D\mathbf{v} \phi \, {\rm d}x + \int_\Omega \mathbb{S}(D\mathbf{v}):(D\mathbf{u}-D\mathbf{v}) \phi \, {\rm d}x.
    \ea
    Combining these limits yields \eqref{eq:target}.
    \end{proof}
    \end{appendices}
\end{document}